\documentclass{tran-l}

\usepackage{amssymb}
\usepackage{tikz}
\usepackage[hidelinks]{hyperref}
\usepackage[capitalise]{cleveref}
\usepackage{subfigure}
\usepackage{enumitem}
\usepackage{ulem}
\usepackage{csquotes}

\usepackage{graphicx}

\newtheorem{theorem}{Theorem}[section]
\newtheorem{lemma}[theorem]{Lemma}

\theoremstyle{definition}
\newtheorem{definition}[theorem]{Definition}
\newtheorem{example}[theorem]{Example}

\newtheorem{proposition}[theorem]{Proposition}

\theoremstyle{remark}
\newtheorem{remark}[theorem]{Remark}

\numberwithin{equation}{section}

\usepackage{comment}
\usepackage{float}

\definecolor{JungleGreen}{rgb}{0.16, 0.67, 0.53}

\def\zh{\hat{0}}
\def\oneh{\hat{1}}

\begin{document}

\title[Equivalences and Distinctions in Lexicographic Shellability]{Equivalences and Distinctions in Lexicographic Shellability of Posets}


\author{Stephen Lacina}
\address{}
\curraddr{}
\email{slacina@truman.edu}
\thanks{}

\author{Grace Stadnyk$^*$
}
\address{}
\curraddr{}
\email{grace.stadnyk@furman.edu}
\thanks{$^*$Support for this research was
provided by an AMS-Simons Research Enhancement Grant for Primarily Undergraduate Institution Faculty.}


\date{}

\dedicatory{}

\commby{}

\begin{abstract} We present two perhaps surprisingly small 
posets, one graded and one non-graded, that are CC-shellable in the sense of Kozlov and TCL-shellable in the sense of Hersh, but not CL-shellable in the sense of Bj{\"o}rner and Wachs. 
In the spirit of Bj{\"o}rner and Wachs' recursive atom orderings (RAO) and Hersh and Stadnyk's generalized recursive atom orderings (GRAO), we also introduce a notion called recursive first atom sets (RFAS). An RFAS is a set of conditions on the atoms of each interval in a finite bounded poset $P$ that are necessary for CC-shellability of $P$ and sufficient for shellability of $P$. We also prove that under an extra condition, $P$ is CC-shellable if and only if it admits an RFAS, in the same way that RAOs provide a reformulation of CL-shellability. 
\end{abstract}

\maketitle

\section{Introduction}\label{sec:intro}

\begin{center}
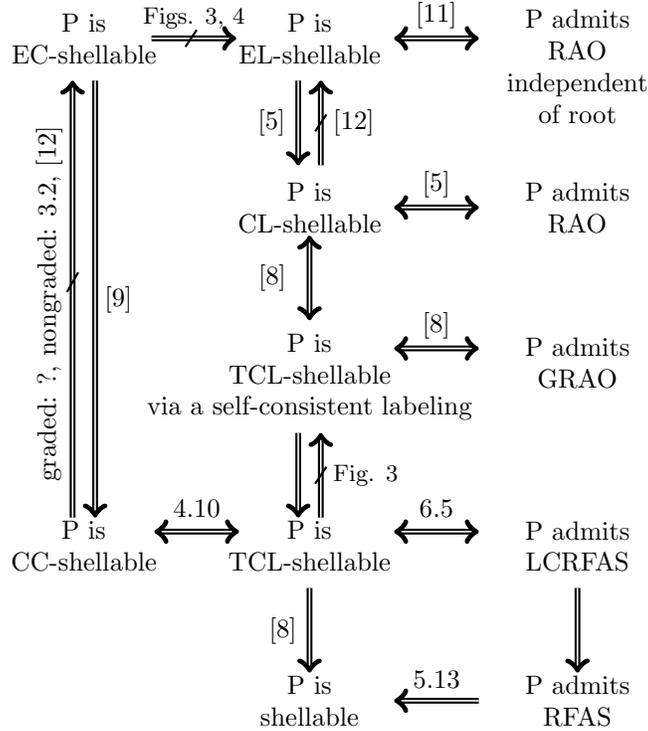
\begin{figure}
\begin{tikzpicture}[scale=0.75]
\draw (2, 0.25) node[align=center,minimum size=3cm] {P is \\ CC-shellable };
\draw (6, 0.25) node[align=center,minimum size=3cm] {P is \\ TCL-shellable };
\draw (6, 3.25) node[align=center,minimum size=3cm] {P is \\ TCL-shellable\\via a self-consistent labeling};
\draw (6, -2.5) node[align=center,minimum size=3cm] {P is \\shellable};
\draw (6, 6.25) node[align=center,minimum size=3cm] {P is \\ CL-shellable };
\draw (2, 9.25) node[align=center,minimum size=3cm] {P is \\ EC-shellable};
\draw (6, 9.25) node[align=center,minimum size=3cm] {P is \\ EL-shellable}; 
\draw (10.75, 0.25) node[align=center,minimum size=3cm] {P admits \\LCRFAS}; 
\draw (10.75, -2.5) node[align=center,minimum size=3cm] {P admits \\RFAS}; 
\draw (10.75, 3.5) node[align=center,minimum size=3cm] {P admits \\GRAO}; 
\draw (10.75, 6.25) node[align=center,minimum size=3cm] {P admits \\ RAO}; 
\draw (10.75, 8.75) node[align=center,minimum size=3cm] {P admits \\ RAO \\ independent \\ of root}; 

\draw [thick, double,<->] (3.25, 0.5)--node [align=center, minimum size=2cm, midway, above=-2em] {\ref{thm: tcl iff cc}}(4.75, 0.5); 
\draw [thick, double, ->](5.8, 8.5)--node[midway, left=0.1em]{\cite{bw}} (5.8, 7);
\draw [thick, double, ->](10.75, -0.5)--node[midway, left=0.1em]{} (10.75, -2);
\draw [thick, double, ->](6, -0.5)--node[midway, left=0.1em]{\cite{HerStad}} (6, -2);
\draw [thick, double, ->](9, -2.5)--node[midway, above=0.1em]{\ref{thm:RFAS implies shellable}} (7.5, -2.5);
\draw [thick, double, <-](6.2, 8.5)--node[midway, right=0.1em]{\cite{liclnotel}}(6.2, 7);
\draw [thick](6.1, 7.6)--(6.3, 7.9);
\draw [thick, double, <->](6, 5.75)--node[midway, left=0.5em]{\cite{HerStad}}(6, 4.25);
\draw [thick, double, ->](5.8, 2.25)--node[midway, left=0.1em]{}(5.8, 0.75);
\draw [thick, double, <-](6.2, 2.25)--node[midway, right=0.1em]{\small{Fig.} \ref{fig: ccnotclex}}(6.2, 0.75);
\draw [thick](6.1, 1.35)--(6.3,1.65);
\draw [thick, double,<->] (7.5, 9.25)--node [align=center, minimum size=2cm, midway, above=-2em] {\cite{liraoindep}}(9, 9.25); 
\draw [thick, double,<->] (7.5, 6.25)--node [align=center, minimum size=2cm, midway, above=-2em] {\cite{bw}}(9, 6.25); 
\draw [thick, double,<->] (7.5, 0.5)--node [align=center, minimum size=2cm, midway, above=-2em] {\ref{thm: tcl iff RFAS}}(9, 0.5); 
\draw [thick, double,<->] (7.5, 3.75)--node [align=center, minimum size=2cm, midway, above=-2em] {\cite{HerStad}}(9, 3.75); 
\draw [double, thick, <-] (4.7, 9.25)--node[align=center, minimum size=2cm, midway, above=-2.1em] {\small{Figs.} \ref{fig: ccnotclex}, \ref{nongradedccnotclex}}(3.2, 9.25);
\draw [thick](3.8, 9.1)--(4, 9.4);

\draw [double, thick, ->] (1.8, 0.75)--node [align=center, minimum size=2cm, midway,  above=-2em, sloped] {graded: ?, nongraded: \ref{rmk:cc not ec li}, \cite{liclnotel} } (1.8, 8.5);
\draw [thick](1.7, 4.75)--(1.9, 5.1);
\draw [double, thick, ->] (2.2, 8.5)--node[align=center, minimum size=2cm, midway, right=-2em] {\cite{kozlov}}(2.2, 0.75);
\end{tikzpicture}
\vspace{-0.5in}
\caption{Relationships between different notions of lexicographic shellability and where the relationships are proved in this paper and elsewhere in the literature.}
\label{fig:implications}
\end{figure}
\end{center}

Since the 1980s, lexicographic shellability has been a popular and powerful tool to establish poset shellability, a property with important topological, combinatorial, and algebraic implications. These implications stem largely from the relationship between a poset $P$ and an associated simplicial complex, $\Delta(P)$, called the order complex of $P$. For instance, $\Delta(P)$ is naturally associated with a commutative ring called the Stanley-Reisner ring. If $P$ is shellable, then $\Delta(P)$ has the homotopy type of a wedge of spheres and the Stanley-Reisner ring is Cohen-Macaulay. Poset shellability, particularly lexicographic shellability, also has applications in other mathematical contexts; many of these arise when considering the shellability of specific posets. For instance, the order complex of the subgroup lattice of a finite group is shellable if and only if the group is solvable (\cite{shareshian2001}), and a lexicographic shelling of a poset associated to a specific subspace arrangement was used to prove a complexity theory lower bound in \cite{lineardecistreestopboundsbjornerlovaszyao1992}. 

Lexicographic shellability is a family of methods that involves labeling the cover relations of a poset $P$ in a way that yields a shelling order of $\Delta(P)$. In particular, these labelings induce a sequence of labels for each maximal chain of $P$. As the faces of $\Delta(P)$ are the chains of $P$, when these label sequences are lexicographically ordered, the corresponding order on maximal chains gives a shelling order of $\Delta(P)$. Bj{\"o}rner introduced the first of these methods, called EL-labeling, in \cite{bjornercm}. At the heart of how EL-labelings produce shellings are pairs of adjacent cover relations where the corresponding pair of labels increase from bottom to top (ascents) or do not increase from bottom to top (descents). 

EL-labelings have been generalized in several ways in the past few decades giving other types of lexicographic shellings. Bj{\"o}rner and Wachs introduced CL-labelings, which allows for edge labels to depend on the path taken from the bottom of the poset (the root) to the edge being labeled (\cite{BW82}). These types of labelings are called chain-edge labelings. As with EL-labelings, the heart of the matter are ascents and descents in the label sequences of maximal chains. 

Another of Bj{\"o}rner and Wachs' significant contributions to the theory of lexicographic shellability was a recursive perspective first introduced in \cite{bw}. They gave a formulation of CL-shellability called recursive atom ordering (RAO), which consists of the specification of total orderings of atoms of upper intervals for a bounded poset. These atom orderings are recursive in that for nested intervals, the respective atom orderings agree with each other in a particular sense. Hersh and Stadnyk introduced generalized recursive atom orderings (GRAOs) in \cite{HerStad} as a more flexible version of RAOs that is nonetheless sufficient for CL-shellability. 

Kozlov introduced a notion of lexicographic shellability called CC-shellability in \cite{kozlov}. His method amounts to showing that actual ascents and descents are not strictly necessary to produce a shelling order from an edge or chain-edge labeling. Hersh then introduced in \cite{hersh} a slightly different version of lexicographic shellability, what has come to be known as TCL-shellability, which similarly relies on a broader notion of ascents and descents. More specifically, for a particular labeling, it is enough to be able to identify pairs of adjacent cover relation that behave like ascents (resp. descents) though the label sequences on these chains may not actually be increasing (resp. not increasing). Hersh calls these pairs of adjacent cover relations topological ascents (resp. topological descents). 

 There have been several results exploring the relationships between these different formulations of lexicographic shellability. For example, Vince and Wachs in \cite{vincewachsshellnonlexshell} and Walker in \cite{walkershellnonlexshell} constructed posets that are shellable but not lexicographically shellable. \cref{fig:implications} illustrates the known relationships among these different versions of lexicographic shellability and their different formulations, including the relationships known from previous work and those proven in this paper.

In \cref{thm: tcl iff cc}, we show that Kozlov's CC-shellability is equivalent to Hersh's TCL-shellability. While it is clear that any CC-labeling is also a TCL-labeling, we construct a CC-labeling for a TCL-shellable poset by introducing a relabeling technique that is intuitive in spirit but quite technical in detail. 
This relabeling technique further plays a central role in the proof of our recursive formulation of CC-shellability.


Hersh and Stadnyk in \cite{HerStad} 
 showed that certain CC-shellable posets, those that they call ``self-consistent" CC-shellable posets, are also CL-shellable. They additionally asked the question are all CC-shellable posets CL-shellable? In \cref{thm:cc but not cl posets}, we show that CC-shellability (and thus TCL-shellability) and CL-shellability, are actually distinct. In particular, we construct quite small graded and non-graded posets (see Figures \ref{fig: ccnotclex}, \ref{nongradedccnotclex}) that are CC-shellable but not CL-shellable. Our examples, together with Li's examples of posets that are CL-shellable but not EL-shellable (\cite{liclnotel}), help to establish non-equivalence of (and in fact, strict containment among) various notions of lexicographic shellability.  
 
Building on Bj{\"o}rner and Wachs' recursive perspective on lexicographic shellability, we introduce in \cref{def:RFAS} a new recursive method for shelling a poset called recursive first atom sets (RFAS). Notably, there exist posets (see \cref{ex:RFAS not LCRFAS}) that admit an RFAS but for which there is no labeling that could yield any of the shellings that arise from the RFAS. Thus we specify a compatibility condition in \cref{def:LCRFAS}---when an RFAS satisfies this condition we call it an LCRFAS---and show in \cref{thm: tcl iff RFAS} that a bounded poset admits an LCRFAS if and only if it is CC-shellable. 

As the names suggest, the essence of our recursive formulations is that we only need specify the first atom in every rooted interval rather than specifying a total order on all atoms of upper rooted intervals as in RAOs and GRAOs. In this way, for a particular interval, RFASs and LCRFASs require much less data than RAOs and GRAOs. There is another important contrast between RAOs, GRAOs and LCRFASs that illustrates that there is still substantial subtlety to LCRFASs. In RAOs and GRAOs, the atom ordering depends on the root to the bottom of an interval but it does not depend on the top element of the interval. On the other hand, the first atoms in an LCRFAS can depend on both the root to the bottom of an interval and the top element of the interval.

The structure of this paper is as follows: \cref{sec:background} provides the necessary background on posets and lexicographic shellability. In \cref{sec:cc not cl}, we give two examples of posets that are CC-shellable and prove in \cref{thm:cc but not cl posets} why these examples cannot be CL-shellable by showing neither can admit an RAO. In \cref{sec:TCL equiv CC}, we show that CC-shellability and TCL-shellability are equivalent in \cref{thm: tcl iff cc} by providing a method for taking a TCL-shellable poset and relabeling it with a CC-labeling. 
In \cref{sec:RFAS} we define RFASs and  prove that an RFAS implies shellability in \cref{thm:RFAS implies shellable}. In Section \ref{sec:lcrfas}, we introduce a LCRFASs and show that an LCRFAS is equivalent to CC-shellability in \cref{thm: tcl iff RFAS}; the proof relies on the relabeling method used in the proof of \cref{thm: tcl iff cc}. Lastly, \cref{sec:open questions} discusses some open questions.

\section{Background}\label{sec:background}
A poset $(P, \leq)$ is a set, together with an order relation $\leq$ that is reflexive, transitive, and antisymmetric. We will often refer to a poset $(P, \leq)$ as $P$ when the order relation $\leq$ is understood. If $x < y$ in $P$ and there exists no $z$ such that $x < z < y$, then we call $x < y$ a \textbf{cover relation}. We will usually denote a cover relation by $x \lessdot y$. If $x \leq y$ in $P$, then the \textbf{interval} $[x, y]$ is the poset $[x, y]=\{z: x \leq z \leq y\}$ with order relations inherited from $P$. The poset $P$ is \textbf{bounded} if there exist in $P$ a unique maximal element, denoted $\hat{1}$, and a unique minimal element, denoted $\hat{0}$. In this paper, we consider only bounded posets. The elements covering $\hat{0}$ are called the \textbf{atoms} of $P$ while the elements covered by $\hat{1}$ are call the \textbf{coatoms} of $P$. A \textbf{chain} of $P$ is a totally ordered subset of $P$. We will denote a chain both by $x_1 < x_2 < \ldots < x_k$ and by $\{x_1, x_2, \ldots, x_k\}$. The \textbf{length} of a chain $c$ is one less than the number of elements in $c$. A chain is \textbf{maximal} if it is not properly contained in any other chain of $P$. A chain is \textbf{saturated} if it is a maximal chain of an interval of $P$. If all maximal chains in $P$ have the same length, then $P$ is said to be \textbf{graded}. Otherwise, $P$ is \textbf{nongraded}. A \textbf{rooted interval} $[x, y]_r$ is the interval $[x, y]$ together with a maximal chain $r$ in $[\hat{0}, x]$; $r$ is called the \textbf{root}. A \textbf{rooted cover relation} is a rooted interval $[x, y]_r$ where $x \lessdot y$. If $m$ is a chain containing $x$ in $P$, we will denote by $m^x$ the subchain of $m$ given by $\{y: y \in m, y \leq x\}$. 

We will assume the reader has a basic understanding of simplicial complexes, though they may choose to refer to \cite{Wachs} for more details that are helpful in the context of poset shellability. We simply remind the reader that every poset $P$ can be associated with a simplicial complex $\Delta(P)$, called the \textbf{order complex of $P$}, whose $k-$faces are the chains of length $k$ in $P$. A simplicial complex $\Delta$ is said to be \textbf{shellable} if its facets can be ordered $F_1, F_2, \ldots ,F_s$ such that for $1 < j \leq s$, $\overline{F_j} \cap \left(\cup_{1 \leq i \leq {j-1}} \overline{F_i}\right)$ is pure and $(\text{dim}(F_j)-1)$-dimensional. A poset is shellable if its order complex is shellable. 

One common method for proving a poset is shellable is by labeling the cover relations of the poset in a particular way, which then induces a  total order on the maximal chains of the poset and hence the facets of the order complex; this family of methods is called \textbf{lexicographic shellability}. We now provide a brief overview of these various methods. For more complete but concise references on different types lexicographic shellability see \cite{HerStad} and \cite{Wachs}. 

The simplest version of lexicographic shellability arises from particular kinds of edge labelings called EL-labelings, which were introduced by Bj{\"o}rner in \cite{bjornercm}; a version for nongraded posets was introduced in \cite{non-pure1}. An \textbf{edge labeling} of a poset $P$ is a map $\lambda$ from the set of cover relations of $P$ to the elements of another poset $\Lambda$.
For any edge labeling $\lambda$ of $P$ and any saturated chain $c$
in $P$, the \textbf{label sequence} $\lambda(c)$ associated to $c$ by $\lambda$ is the sequence obtained by reading the labels assigned by $\lambda$ to the cover relations in $c$ from the bottom of the chain to the top. An edge labeling induces a partial order on the maximal chains of the poset by taking the lexicographic order on label sequences on maximal chains. (Note when we say \textbf{lexicographic order} we mean dictionary order; in particular $aa$ comes before $aaa$ in lexicographic order.)
An \textbf{EL-labeling} of the bounded poset $P$ is an edge labeling of $P$ satisfying the requirement that in any interval, exactly one maximal chain has a strictly increasing label sequence, and this label sequence comes lexicographically earlier than the label sequence on all other maximal chains in the interval. Any poset that admits an EL-labeling is said to be \textbf{EL-shellable} and is shellable \cite{bjornercm}. 

More flexibility in labeling is afforded by instead labeling cover relations in chains, which allows one to label the same cover relation differently depending on which chain is being labeled. This more flexible labeling scheme is called a CE-labeling. A \textbf{CE-labeling} (chain-edge labeling) of a bounded poset $P$ is a map $\lambda$ from the set of rooted cover relations of $P$ to a poset $\Lambda$. For a CE-labeling $\lambda$, we will let $\lambda(r, x, u)$ be the label on the rooted cover relation $[x, u]_r$. For the rooted interval $[x, y]_r$ and a maximal chain $c$ given by $x \lessdot x_1 \lessdot \ldots \lessdot x_t=y$ in $[x, y]_r$, the \textbf{label sequence associated to $c$} is $\lambda(r, x, x_1), \lambda(r \cup \{x_1\}, x_1, x_2), \ldots, \lambda(r \cup \{x_1, \ldots, x_{t-1}\}, x_{t-1}, x_t)$. We will let $\lambda(r)$ denote this label sequence for a root $r$. A CE-labeling $\lambda$ is a \textbf{CL-labeling} if for every rooted interval $[x, y]_r$ in $P$, there is a unique maximal chain $c$ with a strictly increasing label sequence in $[x, y]_r$, and the label sequence associated to $c$ lexicographically precedes the label sequence associated to all other maximal chains in $[x, y]_r$. If the poset $P$ admits such a labeling, it is said to be \textbf{CL-shellable} and it is shellable \cite{BW82}. In \cite{bw}, Bj\"{o}rner and Wachs introduced the notion of recursive atom ordering, and showed that a poset admits a recursive atom ordering if and only if it is CL-shellable. 

\begin{definition}[\cite{bw}]
    A bounded poset $P$ admits a \textbf{recursive atom ordering} (RAO) if the length of the longest chain in $P$ is 1, or if the length of the longest chain in $P$ is greater than 1 and there is an ordering $a_1, a_2, \ldots, a_t$ of the atoms of $P$ satisfying the following: 
    \begin{enumerate}
        \item For $j=1, 2, \ldots t$, $[a_j, \hat{1}]$ admits a recursive atom ordering such that the atoms of $[a_j, \hat{1}]$ covering some $a_i$ for $i < j$ come first. 
        \item For $i < j$, if $a_i, a_j < y$ for some $y \in P$, then there exists some $k$ and some $z\in P$ where $k < j$, $a_j \lessdot z \leq y$, and $a_k < z$. 
    \end{enumerate}
\end{definition}

In \cite{hersh}, Hersh introduced the language of topological ascents and descents to describe a more general version of lexicographic shellability. We will define both CC-labelings and TCL-labelings using this language, though CC-labelings were introduced first by Kozlov in \cite{kozlov} without explicitly using the language of ascents and descents. Suppose that $\lambda$ is a CE-labeling. Then a \textbf{topological ascent} is a pair of rooted cover relations $[u, v]_r$, $[v, w]_{r \cup v}$ in $P$ for which the pair of labels $\lambda(r,u, v), \lambda(r\cup \{v\},v, w)$ lexicographically precedes all other label sequences on maximal chains in $[u, w]_r$.  If this  pair of rooted cover relations is not a topological ascent, then it is called a \textbf{topological descent}. If a chain $c$ in $[x, y]_r$ consists entirely of topological ascents, we say that $c$ is \textbf{topologically ascending}. 
\begin{definition}[\cite{kozlov}]
    A CE-labeling is a \textbf{CC-labeling} if every rooted interval $[x, y]_r$ has a unique topologically ascending maximal chain, the label sequences for the maximal chains in $[x, y]_r$ are all distinct, and no label sequence is the prefix for any other. If $P$ admits a CC-labeling, then $P$ is called \textbf{CC-shellable}.
\end{definition}

Kozlov showed that any poset that is CC-shellable is in fact shellable. He also defined EC-labelings in \cite{kozlov}, which are CC-labelings whose labels on cover relations do not vary depending on the root, in the same way that EL-labelings can be viewed as CL-labelings whose labels do not depend on roots. 

\begin{definition}[\cite{hersh}]
    A CE-labeling is a \textbf{TCL-labeling} if every rooted interval $[x, y]_r$ has a unique, topologically ascending maximal chain. If $P$ admits a TCL-labeling, $P$ is said to be \textbf{TCL-shellable}.
\end{definition} 

Hersh showed that TCL-shellable posets are shellable in \cite{hersh}. Note that there exist TCL-labelings that are not CC-labelings because of the requirement that CC-labelings have distinct label sequences. There are some important consequences of the added flexibility afforded by TCL-labelings which we highlight in the following two remarks. 

\begin{remark}
    \label{rmk: descent sets}
    Notice that if $\lambda$ is a TCL-labeling of the poset $P$ and $\lambda'$ is another CE-labeling of $P$ with a topological descent at $u \lessdot v \lessdot w$ exactly when $u \lessdot v \lessdot w$ is a topological descent with respect to $\lambda$, then $\lambda'$ must also be a TCL-labeling of $P$.
\end{remark}

\begin{remark}\label{rmk:lin ext of label set preserves tcl}
Let $P$ be a finite bounded poset with TCL-labeling $\lambda$ taking labels in a poset $\Lambda$. Observe that ordering the labels by any linear extension of $\Lambda$ preserves the property of $\lambda$ being a TCL-labeling. This is because the topologically ascending maximal chain in any rooted interval must strictly lexicographically precede any other maximal chain in the interval and this is preserved after taking a linear extension of $\Lambda$. Thus, we may assume that the labels of all TCL-labelings are integers. Kozlov observes this for CC-labelings in \cite{kozlov}. Note that this does not hold for either EL- or CL-labelings.
\end{remark}

\begin{figure}
\begin{tikzpicture}
    \node [below] at (0,0) {}; 
\draw [fill] (0,0) circle [radius=0.05]; 
\node [left] at (-1,1) {}; 
\draw [fill] (-1,1) circle [radius=0.05]; 
\node [right] at (1,1) {}; 
\draw [fill] (1,1) circle [radius=0.05]; 
\node [left] at (-1,2) {}; 
\draw [fill] (-1,2) circle [radius=0.05]; 
\node [right] at (1,2) {}; 
\draw [fill] (1,2) circle [radius=0.05]; 
\node [above] at (0,3) {}; 
\draw [fill] (0,3) circle [radius=0.05];
\draw (1,1)--node[near start, below]{\tiny 1}(-1,2);
\draw (0,0)--node[midway, left]{\tiny 1}(-1,1)--node[midway, left]{\tiny 2}(-1,2)--node[midway, left]{\tiny 3}(0,3);
\draw (0,0)--node[midway, right]{\tiny 3}(1,1)--node[midway, right]{\tiny 2}(1,2)--node[midway, right]{\tiny 1}(0,3);
\draw (-1,1)--node[near start, below]{\tiny 3}(1,2);
\end{tikzpicture}
\hskip 0.5in
\begin{tikzpicture}
 \node [below] at (0,0) {}; 
\draw [fill] (0,0) circle [radius=0.05]; 
\node [left] at (-1,1) {}; 
\draw [fill] (-1,1) circle [radius=0.05]; 
\node [right] at (1,1) {}; 
\draw [fill] (1,1) circle [radius=0.05]; 
\node [left] at (-1,2) {}; 
\draw [fill] (-1,2) circle [radius=0.05]; 
\node [right] at (1,2) {}; 
\draw [fill] (1,2) circle [radius=0.05]; 
\node [above] at (0,3) {}; 
\draw [fill] (0,3) circle [radius=0.05];
\draw (1,1)--node[draw,inner sep=1pt, near start, below=2pt]{\tiny 2}(-1,2);
\draw (0,0)--node[style={draw,circle}, inner sep=0.5pt, midway, left=3pt]{\tiny 1}(-1,1)--node[style={draw,circle}, inner sep=0.5pt, midway, left=2pt]{\tiny 2}(-1,2)--node[style={draw,circle}, inner sep=0.5pt, midway, left=3pt]{\tiny 3} node[midway, right=4pt, draw, inner sep=1pt]{\tiny 1}(0,3);
\draw (0,0)--node[draw, inner sep=1, midway, right=4pt]{\tiny 3}(1,1)--node[draw, inner sep=1pt, midway, right=3pt]{\tiny 1}(1,2)--node[draw, inner sep=1pt, midway, right=6pt]{\tiny 2} node [style={draw, circle}, midway, right=4pt]{}(0,3);
\draw (-1,1)--node[style={draw,circle}, inner sep=0.5pt, near start, below=2pt]{\tiny 3}(1,2);
\end{tikzpicture}
\hskip 0.5in
\begin{tikzpicture}
 \node [below] at (0,0) {}; 
\draw [fill] (0,0) circle [radius=0.05]; 
\node [left] at (-1,1) {}; 
\draw [fill] (-1,1) circle [radius=0.05]; 
\node [right] at (1,1) {}; 
\draw [fill] (1,1) circle [radius=0.05]; 
\node [left] at (-1,2) {}; 
\draw [fill] (-1,2) circle [radius=0.05]; 
\node [right] at (1,2) {}; 
\draw [fill] (1,2) circle [radius=0.05]; 
\node [above] at (0,3) {}; 
\draw [fill] (0,3) circle [radius=0.05];
\draw (1,1)--node[near start, below]{\tiny 3}(-1,2);
\draw (0,0)--node[midway, left]{\tiny 1}(-1,1)--node[midway, left]{\tiny 2}(-1,2)--node[midway, left]{\tiny 4}(0,3);
\draw (0,0)--node[midway, right]{\tiny 1}(1,1)--node[midway, right]{\tiny 4}(1,2)--node[midway, right]{\tiny 5}(0,3);
\draw (-1,1)--node[near start, below]{\tiny 5}(1,2); 
\end{tikzpicture}

    \caption{A poset $P$ with an EL-labeling (left), a CL-labeling that is not an EL-labeling (middle), and a CC-labeling that is neither an EL-labeling nor a CL-labeling (right), as described in Example \ref{ex: labelings}.}
    \label{fig: labelings} 
\end{figure}
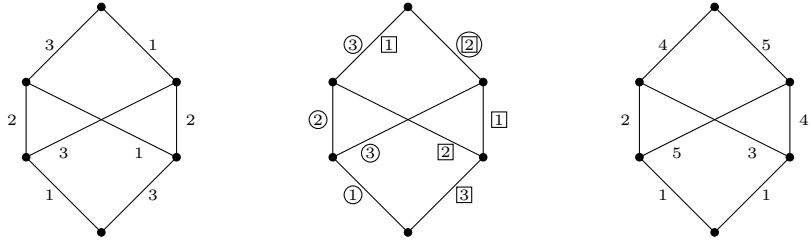

\begin{example}
\label{ex: labelings}
Figure \ref{fig: labelings} gives three different labelings on a poset $P$. The poset on the left of the figure is labeled with an EL-labeling. The poset in the middle of the figure is labeled with a CL-labeling that is not an EL-labeling, since the label on the top left cover relation in the poset depends on the root leading from the bottom of the poset to the cover relation. Specifically, the circled labels correspond to cover relations within chains that include the leftmost atom of the poset, whereas the boxed labels give the labels on cover relations within chains that include the rightmost atom of the poset. The poset on the right of the figure is labeled by a CC-labeling that is neither an EL-labeling nor a CL-labeling, since there exist intervals in the poset where several maximal chains in the interval have  increasing label sequences. Note that the lexicographic ordering on maximal chains induced by each of the labelings in the figure is different.
\end{example}

In \cite{HerStad}, Hersh and Stadnyk introduced the notion of generalized recursive atom ordering:
\begin{definition}[\cite{HerStad}] A poset $P$ admits a \textbf{generalized recursive atom ordering} (or GRAO) if the length of $P$ is 1 or if the length of $P$ is greater than 1 and there exists an ordering $a_1, a_2, \ldots a_t$ of the atoms of $P$ satisfying the following: 
\begin{enumerate}
\item 
\begin{enumerate}
    \item For $1 \leq j \leq t$, $[a_j, \hat{1}]$ admits a GRAO
    \item For atom $a_j$ and $x, w$ such that $a_j \lessdot x \lessdot w$, the following holds when the ordering from (i)(a) is restricted to $[a_j, w]$: either the first atom of $[a_j, w]$ is above some $a_i$ with $i < j$, or no atom of $[a_j, w]$ is above any $a_i$ with $i < j$.
\end{enumerate}
\item For any $y \in P$ and $a_i, a_j < y$ with $i < j$, there exists some $z \in P$ and some atom $a_k \in P$ where $a_j \lessdot z \leq y$, $a_k < z$, and $k < j$. 
\end{enumerate}
\end{definition} 

Furthermore, they showed that a poset admits a GRAO if and only if it is TCL-shellable using a TCL-labeling that is self-consistent. A \textbf{self-consistent} TCL-labeling is a TCL-labeling with the property that whenever $a$ and $b$ are atoms in a rooted interval $[x, y]_r$ and $a$ is in the lexicographically first saturated chain of $[x, y]_r$, any saturated chain containing the atom $a$ in a rooted interval $[x, y']_r$ comes lexicographically earlier than any saturated chain containing $b$ in $[x, y']_r$. In \cite{HerStad}, Hersh and Stadnyk also showed that any GRAO can be transformed into an RAO, thereby showing that any poset that admits a self-consistent TCL-labeling is also CL-shellable.  

\section{Posets that are CC-shellable but not CL-shellable}\label{sec:cc not cl}

Here we provide two examples of posets---one graded and one nongraded---that are TCL-shellable but not CL-shellable as neither admits a recursive atom ordering.

\begin{figure}
\begin{center}
\begin{tikzpicture}[scale=0.75]

\node at (5,0){};
\draw [fill] (5,0) circle [radius=0.075]; 
\node at (0,2){};
\draw [fill] (0,2) circle [radius=0.075];
\node at (5,2){};
\draw [fill] (5,2) circle [radius=0.075];
\node at (10,2){};
\draw [fill] (10,2) circle [radius=0.075];
\node at (-2,8){};
\draw [fill] (-2,8) circle [radius=0.075];
\node at (1,8){};
\draw [fill] (1,8) circle [radius=0.075];
\node at (4,8){};
\draw [fill] (4,8) circle [radius=0.075];
\node at (6,8){};
\draw [fill] (6,8) circle [radius=0.075];
\node at (9.25,8){};
\draw [fill] (9.25,8) circle [radius=0.075];
\node at (12,8){};
\draw [fill] (12,8) circle [radius=0.075];
\node at (-1,12){};
\draw [fill] (-1,12) circle [radius=0.075];
\node at (3, 12){};
\draw [fill] (3,12) circle [radius=0.075];
\node at (7,12){};
\draw [fill] (7,12) circle [radius=0.075];
\node at (11,12){};
\draw [fill] (11,12) circle [radius=0.075];
\node at (5,14){};
\draw [fill] (5,14) circle [radius=0.075];


\draw (5,0)--node [near start, circle, fill=white, inner sep=0.5pt,minimum size=1pt] {\footnotesize \textbf{1}}node [midway, circle, fill=white, inner sep=0.5pt,minimum size=1pt] {\footnotesize (1)} (0,2);
\draw (5,0)--node [near start, circle, fill=white, inner sep=0.5pt,minimum size=1pt] {\footnotesize \textbf{1}} node [midway, circle, fill=white, inner sep=0.5pt,minimum size=1pt] {\footnotesize (2)}(5,2);
\draw (5,0)--node [near start, circle, fill=white, inner sep=0.5pt,minimum size=1pt] {\footnotesize \textbf{1}} node [midway, circle, fill=white, inner sep=0.5pt,minimum size=1pt] {\footnotesize (4)}(10,2);

\draw (0,2)--node [pos=.35, circle, fill=white, inner sep=0.5pt,minimum size=1pt] {\footnotesize \textbf{1}} node [pos=.5, circle, fill=white, inner sep=0.5pt,minimum size=1pt] {\footnotesize (2)}(-2,8);
\draw (0,2)--node [pos=.3, circle, fill=white, inner sep=0.5pt,minimum size=1pt] {\footnotesize \textbf{9}}node [pos=.45, circle, fill=white, inner sep=0.5pt,minimum size=1pt] {\footnotesize (2)}(1,8);
\draw (0,2)--node [pos=.21, circle, fill=white, inner sep=0.5pt,minimum size=1pt] {\footnotesize \textbf{1}} node [pos=.3, circle, fill=white, inner sep=0.5pt,minimum size=1pt] {\footnotesize (4)}(6,8);
\draw (0,2)--node [very near start, circle, fill=white, inner sep=0.5pt,minimum size=1pt] {\footnotesize \textbf{10}}node [pos=.2, circle, fill=white, inner sep=0.5pt,minimum size=1pt] {\footnotesize (0)}(9.25,8);

\draw (5,2)--node [very near start, circle, fill=white, inner sep=0.5pt,minimum size=1pt] {\footnotesize \textbf{2}}node [near start, circle, fill=white, inner sep=0.5pt,minimum size=1pt] {\footnotesize (0)}(-2,8);
\draw (5,2)--node [very near start, circle, fill=white, inner sep=0.5pt,minimum size=1pt] {\footnotesize \textbf{8}}node [near start, circle, fill=white, inner sep=0.5pt,minimum size=1pt] {\footnotesize (1)}(1,8);
\draw (5,2)--node [very near start, circle, fill=white, inner sep=0.5pt,minimum size=1pt] {\footnotesize \textbf{3}}node [near start, circle, fill=white, inner sep=0.5pt,minimum size=1pt] {\footnotesize (3)}(4,8);
\draw (5,2)--node [very near start, circle, fill=white, inner sep=0.5pt,minimum size=1pt] {\footnotesize \textbf{6}}node [near start, circle, fill=white, inner sep=0.5pt,minimum size=1pt] {\footnotesize (4)}(12,8);

\draw (10,2)--node [very near start, left, circle, fill=white, inner sep=0.5pt,minimum size=1pt] {\footnotesize \textbf{7}}node [near start, circle, fill=white, inner sep=0.35pt,minimum size=.75pt] {\footnotesize (1)}(4,8);
\draw (10,2)--node [pos=.25, circle, fill=white, inner sep=0.5pt,minimum size=1pt] {\footnotesize \textbf{4}}node [pos=.35, circle, fill=white, inner sep=0.35pt,minimum size=.75pt] {\footnotesize (0)}(6,8);
\draw (10,2)--node [pos=.35, circle, fill=white, inner sep=0.5pt,minimum size=1pt] {\footnotesize \textbf{11}}node [midway, circle, fill=white, inner sep=0.5pt,minimum size=1pt] {\footnotesize (5)}(9.25,8);
\draw (10,2)--node [pos=.35, circle, fill=white, inner sep=0.5pt,minimum size=1pt] {\footnotesize \textbf{5}}node [midway, circle, fill=white, inner sep=0.5pt,minimum size=1pt] {\footnotesize (3)}(12,8);

\draw (-2,8)--node [near start, circle, fill=white, inner sep=0.5pt,minimum size=1pt] {\footnotesize \textbf{1}}node [midway, circle, fill=white, inner sep=0.5pt,minimum size=1pt] {\footnotesize (1)}(-1,12);
\draw (-2,8)--node [near start, circle, fill=white, inner sep=0.5pt,minimum size=1pt] {\footnotesize \textbf{2}}node [pos=.4, circle, fill=white, inner sep=0.5pt,minimum size=1pt] {\footnotesize (5)}(3,12);

\draw (1,8)--node [very near start, circle, fill=white, inner sep=0.5pt,minimum size=1pt] {\footnotesize \textbf{1}} node [near start, circle, fill=white, inner sep=0.5pt, minimum size=1pt] {\footnotesize (6)}(3,12); 
\draw (1,8)--node [pos=0.1, circle, fill=white, inner sep=0.5pt,minimum size=1pt] {\footnotesize \textbf{2}}node [pos=.2, circle, fill=white, inner sep=0.5pt,minimum size=1pt] {\footnotesize (0)}(7,12);

\draw (4,8)--node [pos=.08, circle, fill=white, inner sep=0.5pt,minimum size=1pt] {\footnotesize \textbf{2}}node [pos=.19, circle, fill=white, inner sep=0.5pt,minimum size=1pt] {\footnotesize (0)}(-1,12);
\draw (4,8)--node [very near start, circle, fill=white, inner sep=0.5pt,minimum size=1pt] {\footnotesize \textbf{1}}node [pos=.3, circle, fill=white, inner sep=0.5pt,minimum size=1pt] {\footnotesize (4)}(3,12);

\draw (6,8)--node [very near start, circle, fill=white, inner sep=0.5pt,minimum size=1pt] {\footnotesize \textbf{3}} node [pos=.3, circle, fill=white, inner sep=0.5pt,minimum size=1pt] {\footnotesize (3)}(3,12);
\draw (6, 8)--node [pos=0.08, circle, fill=white, inner sep=0.5pt,minimum size=1pt] {\footnotesize \textbf{4}} node [pos=.2, circle, fill=white, inner sep=0.5pt,minimum size=1pt] {\footnotesize (1)}(11,12);

\draw (9.25,8)--node [pos=.1, circle, fill=white, inner sep=0.5pt,minimum size=1pt] {\footnotesize \textbf{1}} node [pos=.2, circle, fill=white, inner sep=0.5pt,minimum size=1pt] {\footnotesize (7)}(3,12);
\draw (9.25,8)--node [pos=0.1, circle, fill=white, inner sep=0.5pt,minimum size=1pt] {\footnotesize \textbf{2}}node [pos=.25, circle, fill=white, inner sep=0.5pt,minimum size=1pt] {\footnotesize (1)}(7,12);

\draw (12,8)--node [very near start, circle, fill=white, inner sep=0.5pt,minimum size=1pt] {\footnotesize \textbf{1}} node [pos=.25, circle, fill=white, inner sep=0.5pt,minimum size=1pt] {\footnotesize (2)}(3,12);
\draw (12,8)--node [very near start, circle, fill=white, inner sep=0.5pt,minimum size=1pt] {\footnotesize \textbf{2}} node [midway, circle, fill=white, inner sep=0.5pt,minimum size=1pt] {\footnotesize (0)}(11,12);

\draw (-1,12)--node [near start, circle, fill=white, inner sep=0.5pt,minimum size=1pt] {\footnotesize \textbf{1}} node [midway, circle, fill=white, inner sep=0.5pt,minimum size=1pt] {\footnotesize (3)}(5,14);
\draw (3,12)--node [near start, circle, fill=white, inner sep=0.5pt,minimum size=1pt] {\footnotesize \textbf{1}} node [midway, circle, fill=white, inner sep=0.5pt,minimum size=1pt] {\footnotesize (1)}(5,14);
\draw (7,12)--node [near start, circle, fill=white, inner sep=0.5pt,minimum size=1pt] {\footnotesize \textbf{1}} node [midway, circle, fill=white, inner sep=0.5pt,minimum size=1pt] {\footnotesize (4)}(5,14);
\draw (11,12)--node [near start,circle, fill=white, inner sep=0.5pt,minimum size=1pt] {\footnotesize \textbf{1}} node [midway, circle, fill=white, inner sep=0.5pt,minimum size=1pt] {\footnotesize (2)}(5,14);

\end{tikzpicture}

\caption{A graded TCL-shellable poset that is not CL-shellable. A TCL-labeling of the poset is given by the bold labels. A dual EL-labeling of the poset is given by the labels in parentheses.
}
\label{fig: ccnotclex}
\end{center}
\end{figure}
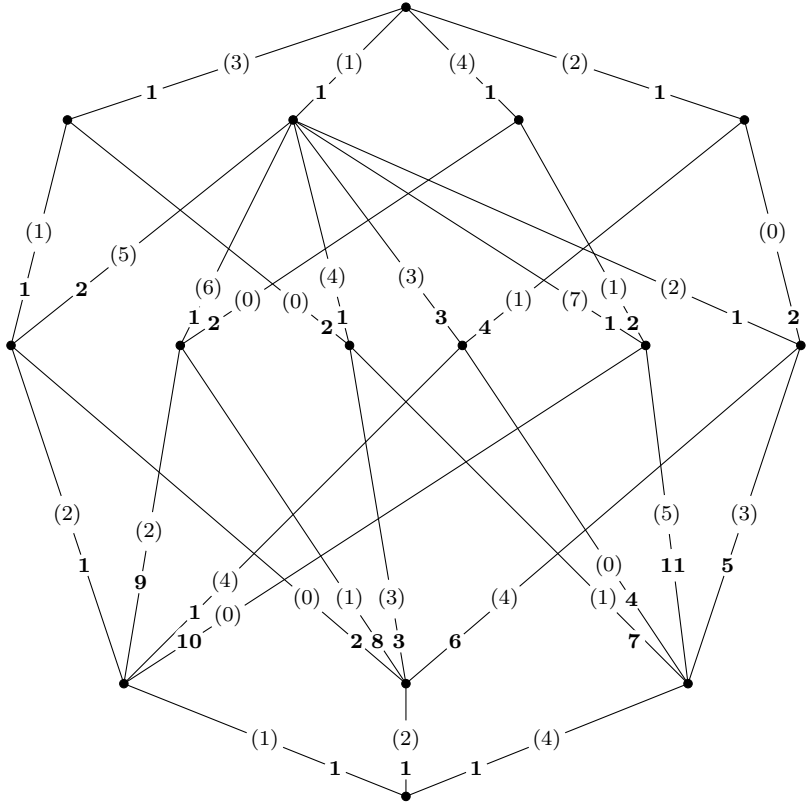

\begin{figure}
\begin{center}
\begin{tikzpicture}[scale=0.75]

\node at (5,0){};
\draw [fill] (5,0) circle [radius=0.05]; 
\node at (7.4,2){$d$};
\node at (7,2){};
\draw [fill] (7,2) circle [radius=0.05];
\node at (2.6,2){$c$};
\node at (3,2){};
\draw [fill] (3,2) circle [radius=0.05];
\node at (7,4){};
\draw [fill] (7,4) circle [radius=0.05];
\node at (3,3){};
\draw [fill] (3,3) circle [radius=0.05];
\node at (2.6,4) {$y'$};
\node at (3,4){};
\draw [fill] (3,4) circle [radius=0.05];
\node at (3,6){};
\draw [fill] (3,6) circle [radius=0.05];
\node at (7,6){};
\draw [fill] (7,6) circle [radius=0.05];
\node at (4.6,6){$y$};
\node at (5,6){};
\draw [fill] (5,6) circle [radius=0.05];

\node at (3,8){};
\draw [fill] (3,8) circle [radius=0.05];
\node at (7,8){};
\draw [fill] (7,8) circle [radius=0.05];

\node at (5,10){};
\draw [fill] (5,10) circle [radius=0.05];


\draw (5,0)--node [midway, circle, fill=white, inner sep=0.5pt,minimum size=1pt] {\small 1}(7,2);
\draw (5,0)--node [midway, circle, fill=white, inner sep=0.5pt,minimum size=1pt] {\small 1}(3,2);

\draw (3,2)--node [midway, circle, fill=white, inner sep=0.5pt,minimum size=1pt] {\small 1}(3,3);
\draw (3,2)--node [near start, circle, fill=white, inner sep=0.5pt,minimum size=1pt] {\small 5}(5,6);

\draw (7,2)--node [midway, circle, fill=white, inner sep=0.5pt,minimum size=1pt] {\small 3}(7,4);
\draw (7,2)--node [near start, circle, fill=white, inner sep=0.5pt,minimum size=1pt] {\small 2}(3,4);

\draw (3,3)--node [midway, circle, fill=white, inner sep=0.5pt,minimum size=1pt] {\small 1}(3,4);

\draw (3,4)--node [midway, circle, fill=white, inner sep=0.5pt,minimum size=1pt] {\small 1}(3,6);

\draw (7,4)--node [near start, circle, fill=white, inner sep=0.5pt,minimum size=1pt] {\small 2}(3,6);
\draw (7,4)--node [midway, circle, fill=white, inner sep=0.5pt,minimum size=1pt] {\small 3}(7,6);
\draw (7,4)--node [midway, circle, fill=white, inner sep=0.5pt,minimum size=1pt] {\small 4}(5,6);

\draw (3,6)--node [midway, circle, fill=white, inner sep=0.5pt,minimum size=1pt] {\small 1}(3,8);

\draw (7,6)--node [near start, circle, fill=white, inner sep=0.5pt,minimum size=1pt] {\small 2}(3,8);
\draw (7,6)--node [midway, circle, fill=white, inner sep=0.5pt,minimum size=1pt] {\small 3}(7,8);

\draw (3,8)--node [midway, circle, fill=white, inner sep=0.5pt,minimum size=1pt] {\small 1}(5,10);

\draw (5,6)--node [near start, circle, fill=white, inner sep=0.5pt,minimum size=1pt] {\small 4}(3,8);
\draw (5,6)--node [very near start, circle, fill=white, inner sep=0.5pt,minimum size=1pt] {\small 5}(7,8);

\draw (7,8)--node [midway, circle, fill=white, inner sep=0.5pt,minimum size=1pt] {\small 2}(5,10);

\end{tikzpicture}
\hskip 0.75in
\begin{tikzpicture}[scale=0.75]

\node at (5,0){};
\draw [fill] (5,0) circle [radius=0.05]; 
\node at (7,2){};
\draw [fill] (7,2) circle [radius=0.05];
\node at (3,2){};
\draw [fill] (3,2) circle [radius=0.05];
\node at (7,4){};
\draw [fill] (7,4) circle [radius=0.05];
\node at (3,3){};
\draw [fill] (3,3) circle [radius=0.05];
\node at (3,4){};
\draw [fill] (3,4) circle [radius=0.05];
\node at (3,6){};
\draw [fill] (3,6) circle [radius=0.05];
\node at (7,6){};
\draw [fill] (7,6) circle [radius=0.05];
\node at (5,6){};
\draw [fill] (5,6) circle [radius=0.05];

\node at (3,8){};
\draw [fill] (3,8) circle [radius=0.05];
\node at (7,8){};
\draw [fill] (7,8) circle [radius=0.05];

\node at (5,10){};
\draw [fill] (5,10) circle [radius=0.05];


\draw (5,0)--node [midway, circle, fill=white, inner sep=0.5pt,minimum size=1pt] {\small 4}(7,2);
\draw (5,0)--node [midway, circle, fill=white, inner sep=0.5pt,minimum size=1pt] {\small 6}(3,2);

\draw (3,2)--node [midway, circle, fill=white, inner sep=0.5pt,minimum size=1pt] {\small 5}(3,3);
\draw (3,2)--node [near start, circle, fill=white, inner sep=0.5pt,minimum size=1pt] {\small 7}(5,6);

\draw (7,2)--node [midway, circle, fill=white, inner sep=0.5pt,minimum size=1pt] {\small 3}(7,4);
\draw (7,2)--node [near start, circle, fill=white, inner sep=0.5pt,minimum size=1pt] {\small 5}(3,4);

\draw (3,3)--node [midway, circle, fill=white, inner sep=0.5pt,minimum size=1pt] {\small 4}(3,4);

\draw (3,4)--node [midway, circle, fill=white, inner sep=0.5pt,minimum size=1pt] {\small 3}(3,6);

\draw (7,4)--node [near start, circle, fill=white, inner sep=0.5pt,minimum size=1pt] {\small 5}(3,6);
\draw (7,4)--node [midway, circle, fill=white, inner sep=0.5pt,minimum size=1pt] {\small 2}(7,6);
\draw (7,4)--node [midway, circle, fill=white, inner sep=0.5pt,minimum size=1pt] {\small 2}(5,6);

\draw (3,6)--node [midway, circle, fill=white, inner sep=0.5pt,minimum size=1pt] {\small 2}(3,8);

\draw (7,6)--node [very near start, circle, fill=white, inner sep=0.5pt,minimum size=1pt] {\small 5}(3,8);
\draw (7,6)--node [midway, circle, fill=white, inner sep=0.5pt,minimum size=1pt] {\small 1}(7,8);

\draw (5,6)--node [near start, circle, fill=white, inner sep=0.5pt,minimum size=1pt] {\small 8}(3,8);
\draw (5,6)--node [very near start, circle, fill=white, inner sep=0.5pt,minimum size=1pt] {\small 3}(7,8);

\draw (3,8)--node [midway, circle, fill=white, inner sep=0.5pt,minimum size=1pt] {\small 1}(5,10);

\draw (7,8)--node [midway, circle, fill=white, inner sep=0.5pt,minimum size=1pt] {\small 4}(5,10);

\end{tikzpicture}

\caption{An nongraded TCL-shellable poset that is not CL-shellable. A TCL-labeling of the poset is shown on the left. A dual EL-labeling of the poset is shown on the right.}
\label{nongradedccnotclex}
\end{center}
\end{figure}

\begin{theorem}\label{thm:cc but not cl posets}
    There exist graded and non-graded posets that are CC-shellable but not CL-shellable.
\end{theorem}

\begin{proof}
    \cref{fig: ccnotclex} and \cref{nongradedccnotclex} show examples of a graded poset $P$ and a nongraded poset $Q$, respectively, that are each TCL-shellable but not CL-shellable. A TCL-labeling is given for each poset (for $P$, this labeling is given by bold labels in \cref{fig: ccnotclex}; for $Q$, the labeling is given on the left of \cref{nongradedccnotclex}). In fact, each of these labelings is a CC-labeling, and even more specifically, an EC-labeling. Both examples fail to be CL-shellable because no atom order satisfies condition 2 of a GRAO. For $P$, observe that for any pair of atoms $a$ and $b$, there is some rank $3$ element $y$ such that $a,b<y$ while no $z$ in $[\zh,y]$ covers both $a$ and $b$. Thus, no pair of atoms can begin a recursive atom ordering of $P$. 
    
    Similarly, for the two atoms $c$ and $d$ of the nongraded poset $Q$, there are elements $y$ and $y'$ such that $c\lessdot y$ and $d\lessdot y'$, but $y$ does not cover $d$ and $y'$ does not cover $c$. Thus, neither $c$ nor $d$ can be the first atom in a recursive atom ordering of $Q$. 
\end{proof}

     While neither poset $P$ given in \cref{fig: ccnotclex} nor poset $Q$ given in \cref{nongradedccnotclex} is CL-shellable, these posets have TCL-labelings as illustrated in their respective figures. In light of this distinction between TCL- and CL-shellable posets and the examples of graded and non-graded posets that are CL-shellable but not EL-shellable provided by Li in  \cite{liclnotel}, one might wonder whether there exist posets that are CC-shellable but not EC-shellable. We immediately see that this is in fact the case, at least for nongraded posets, by considering Li's non-graded example: 
    \begin{remark}\label{rmk:cc not ec li} Replacing ``EL" with ``EC" in the proof of Theorem 2.1 in \cite{liclnotel} shows that the poset is CC-shellable but not EC-shellable. However, the proof technique used in the proof of Theorem 3.1 in \cite{liclnotel} does not clearly go through for Li's graded example because an EC-labeling does not induce an RAO as shown by our \cref{thm:cc but not cl posets}. 
    \end{remark}

    Observe that both $P$ and $Q$ are dual CL-shellable as witnessed by the dual EL-labelings given in \cref{fig: ccnotclex}  and the labeling on the right of \cref{nongradedccnotclex}, respectively. However, we can easily construct posets that are TCL-shellable but neither CL nor dual CL-shellable. Take the ordinal sum of $P$ with its dual $P^*$. Label this ordinal sum with the given TCL-labeling of $P$, the given EL-labeling of $P^*$, and any integer labeling the cover relation between the $\oneh$ of $P$ and the $\zh$ of $P^*$. The resulting labeling is clearly a TCL-labeling (an EC-labeling more specifically) while the resulting poset is graded and neither CL nor dual CL-shellable. The same construction using $Q$, $Q^*$, and their labelings produces a nongraded example.

    Of further note, neither $P$ nor its dual is the face lattice of a regular CW-complex. To see this, recall the following consequence of Bj{\"o}rner's Proposition 3.1 in \cite{Bj84}. In the face poset of a regular CW-complex, every non-trivial open lower interval $(\emptyset,x)$ is homeomorphic to a sphere. Observe that for the element $x\in P$ with rank $3$ and degree $7$, $[\zh,x]$ has three topologically descending chains, namely the saturated chains labeled $(1,6,1)$, $(1,7,1)$, and $(1,9,1)$. Thus, $(\zh,x)$ is homotopy equivalent to a wedge of three $1$-spheres, so $P$ is not the face lattice of a regular CW-complex. On the other hand, for any atom $x$ of $P$ the order complex of $(x,\oneh)$ is shown in \cref{ordercmplxaboveanatom}. This is clearly not homeomorphic to a sphere despite being homotopy equivalent to a $1$-sphere. Thus, the dual of $P$ is also not the face lattice of a regular CW-complex. This shows that $P$ is not CL-shellable for a somewhat different reason than the examples due to Vince and Wachs \cite{vincewachsshellnonlexshell} and Walker \cite{walkershellnonlexshell}.

\begin{figure}[h]
\begin{center}
\begin{tikzpicture}[,scale=0.6]

\node at (-2,6){};
\draw [fill] (-2,6) circle [radius=0.05];
\node at (1,6){};
\draw [fill] (1,6) circle [radius=0.05];
\node at (6,6){};
\draw [fill] (6,6) circle [radius=0.05];
\node at (9,6){};
\draw [fill] (9,6) circle [radius=0.05];
\node at (0,9){};
\draw [fill] (0,9) circle [radius=0.05];
\node at (3,9){};
\draw [fill] (3,9) circle [radius=0.05];
\node at (7,9){};
\draw [fill] (7,9) circle [radius=0.05];
\node at (10,9){};
\draw [fill] (10,9) circle [radius=0.05];

\draw (-2,6)--(0,9);
\draw (-2,6)--(3,9);

\draw (1,6)--(3,9);
\draw (1,6)--(7,9);

\draw (6,6)--(3,9);
\draw (6,6)--(10,9);

\draw (9,6)--(3,9);
\draw (9,6)--(7,9);

\end{tikzpicture}

\caption{The order complex of $(x,\oneh)$.}
\label{ordercmplxaboveanatom}
\end{center}
\end{figure}

\section{TCL-shellable is Equivalent to CC-shellable}\label{sec:TCL equiv CC}

We begin by introducing a labeling technique based on any total order of the maximal chains of a poset. We use this technique later in this section to show that TCL-shellability is equivalent to CC-shellability. Then, in \cref{sec:lcrfas}, we use the same technique to build CE-labelings from certain sets that we will call recursive first atom sets.

While this labeling is somewhat technical, the essence of the labeling is more straightforward. First, we fix a total order on the maximal chains of $P$. We wish to label a rooted cover relation $(r,x\lessdot y)$ by the index of the first maximal chain $m_i$ in this total order that contains both $r$ and $y$. We do this except when a different rooted cover relation $(r,x\lessdot y')$ appears in maximal chains before $m_i$ and again after $m_i$ in the total order, that is, when the chain $m_i$ is ``sandwiched" by two other chains containing the rooted cover relation $(r, x \lessdot y')$. In this case, the distinct rooted cover relations $(r,x\lessdot y)$ and $(r,x\lessdot y')$ receive the same label. 
In any case, a rooted cover relation is always labeled by the position in the total order of some maximal chain which contains the root.

\begin{definition}\label{def:cc label from tcl}
Let $P$ be a finite bounded poset. Fix a total order $\Gamma: m_1,\dots, m_t$ on the maximal chains of $P$. Define a CE-labeling $\lambda$ of $P$ in the following manner. For a root $r$ from $\zh$ to $x$, let $a_1,\dots, a_s$ be the atoms of $[x,\oneh]_r$ in the order in which they first appear with $r$ in a maximal chain in $\Gamma$. Label the rooted edge $x\lessdot a_j$ with label $\lambda(r, x, a_j)$ defined as follows: Let $m_{i_1}$ be the earliest chain of $\Gamma$ containing $r$ and $a_1$. Set $\lambda(r,x,a_1)=i_1$. Say the first $j-1$ rooted edges have been labeled. Let $m_{i_j}$ be the first chain of $\Gamma$ containing $r$ and $a_j$. If there exists an atom $a_h$ with $h<j$ such that $a_h$ and $r$ are contained in maximal chains $m_k$ and $m_l$ with $k<i_j<l$, set $\lambda(r,x,a_{j})=\lambda(r,x,a_{h})$. Otherwise, set $\lambda(r,x,a_j)=i_j$. 

\end{definition}

Throughout the remainder of this section, given a root $r$, we routinely use the notation $i_j$ to denote the index of the earliest maximal chain (in the total order on maximal chains) containing the atom $a_j$ and the root $r$. Two important consequences of Definition \ref{def:cc label from tcl}, both of which are used in proofs later in this section, are provided in the following remarks.

\begin{remark}\label{rmk:first chain label}
    Notice that using \cref{def:cc label from tcl}, a rooted cover relation $x\lessdot y$ with root $r$ can never have a label larger than the index of the first chain in $\Gamma$ that contains $r\cup \{x,y\}$. 
\end{remark}

\begin{remark}\label{rmk:not labeled by your first chain}
    Observe that by \cref{def:cc label from tcl} we have the following: if a rooted cover relation $x\lessdot y$ with root $r$ is not labeled by the position of the first chain $m_j$ in $\Gamma$ containing $r\cup \{x,y\}$, then there is an element $y'$ covering $x$ such that the first chain $m_i$ in $\Gamma$ containing $r\cup \{x,y'\}$ and some chain $m_k$ containing $r\cup \{x,y'\}$ have $i<j<k$.
\end{remark}

The following proposition is used as a key component in many of the proofs that follow, both in this section and in \cref{sec:lcrfas}. In essence, it describes characteristics of the total order $\Gamma$ that force a set of rooted cover relations to share the same label when labeling the poset using Definition \ref{def:cc label from tcl}. It also shows that Definition \ref{def:cc label from tcl} yields a CE-labeling with the property that no two maximal chains in any rooted have the same label sequence. Here and throughout the rest of the paper, we use the notation $m^x$ to refer to the the elements not above $x$ in a chain $m$ when $m$ contains $x$.

\begin{proposition}\label{prop:same cc relable label seqs}
    Let $P$ be a finite bounded poset. Fix a total order $\Gamma: m_1,\dots, m_t$ on the maximal chains of $P$. Let $\lambda$ be the CE-labeling of $P$ determined by $\Gamma$ as in \cref{def:cc label from tcl}. The following then hold: 
    
    \begin{enumerate}
        \item[(i)] Let $a_1,\dots, a_s$ be the atoms of the rooted interval $[x,\oneh]_r$ in the order induced by $\Gamma$ as in \cref{def:cc label from tcl}. If $\lambda(r,x,a_k)=\lambda(r,x,a_l)$ with $k\leq l$, then we have $\lambda(r,x,a_k)=\lambda(r,x,a_j)$ for all $j$ such that $k \leq j \leq l$. 
        
        \item[(ii)] If maximal chains $m$ and $m'$ of the rooted interval $[x, \hat{1}]_r$ contain distinct atoms of $[x, \oneh]$, then the label sequence associated to $m$ is distinct from the label sequence associated to $m'$. 
        Further, if $P$ is not graded, then for any distinct maximal chains $c$ and $c'$ of $[x,y]_r$, the label sequence associated to $c$ is not a prefix of the label sequence associated to $c'$. 
    \end{enumerate}
\end{proposition}

\begin{proof}
   We show (i) by induction on $l-k$. If 
    $l-k=0$, there is nothing to check. Suppose that for some $n\geq 1$ whenever $l-k<n$, we have that for all $j$ such that $k \leq j \leq l$, $\lambda(r,x,a_k)=\lambda(r,x,a_j)$. Now suppose $l-k=n$.     
    Since $n\geq 2$, there is a smallest $j'$ 
    such that $j'<l$ and $r\cup a_{j'}$ is contained in some $m_p$, where 
    $i_{j'} < i_l < p$.
    Thus, we have $\lambda(r,x,a_{j'})= \lambda(r,x,a_l)$ by Definition \ref{def:cc label from tcl}. Then since $\lambda(r,x,a_l)=\lambda(r,x,a_k)$, we have $\lambda(r,x,a_{j'})=\lambda(r,x,a_k)$ as well. 
    
    If $j'\leq k$, 
    then for all $j$ such that $k<j<l$ we have
    $i_{j'}\leq i_{k} < i_j <i_l < p$. We therefore have $\lambda(r,x,a_k)=\lambda(r,x,a_j)$ by \cref{def:cc label from tcl} for all $j$ such that $k \leq j\leq l$. 
    
    Otherwise, $k<j'<l$. Since $j'-k < l-k=n$ and $\lambda(r,x,a_{j'})=\lambda(r,x,a_k)$, the induction hypothesis says that $\lambda(r,x,a_k)=\lambda(r,x,a_j)$ for all $j$ such that $k \leq j \leq j'$. And since $l-j'<l-k=n$ and $\lambda(r,x,a_{j'}) = \lambda(r,x,a_l)$, the induction hypothesis says that $\lambda(r,x,a_j)=\lambda(r,x,a_l)$ for all $j$ such that $j' \leq j \leq l$. Therefore, $\lambda(r,x,a_k)=\lambda'(r,x,a_j)$ for all $k \leq j \leq l$. 

    For (ii), observe the following. First, by \cref{def:cc label from tcl}, a cover relation with root $r$ is labeled by the position in $\Gamma$ of some maximal chain that contains $r$. Let $x\lessdot a_1\lessdot b_1$ be contained in $m$ and $x\lessdot a_2\lessdot b_2$ be contained in $m'$ with $a_1\neq a_2$. 
    The label $\lambda(r\cup a_1,a_1,b_1)$ is the position in $\Gamma$ of some maximal chain which contains $r\cup a_1$ while the label $\lambda(r\cup a_2,a_2,b_2)$ is the position in $\Gamma$ of some maximal chain which contains $r\cup a_2$. No maximal chain of $P$ contains both $r\cup a_1$ and $r \cup a_2$ because $a_1$ and $a_2$ are distinct. Therefore, $\lambda(r\cup a_1,a_1,b_1) \neq \lambda(r\cup a_2,a_2,b_2)$. This in turn implies that the label sequence associated to $m$ is distinct from the label sequence associated to $m'$. 
    
    Now for any distinct maximal chains $c$ and $c'$ of a rooted interval $[x,y]_r$, there is some minimal $z<y$ such that $c^z=c'^z$ but $c$ and $c'$ contain distinct atoms of $[z,y]$. Then by the previous argument the label sequence associated to $c$ is not a prefix of the label sequence associated to $c'$. 
\end{proof}

We now turn to showing that any poset that admits a TCL-labeling also admits a CC-labeling in \cref{lem:lex cc from tcl gives shelling}. We first make an observation about initial sections of label sequences that arise from TCL-labelings. It is used in the two upcoming technical propositions concerning topological descents and their preservation under a relabeling given by Definition \ref{def:cc label from tcl}. Both propositions are necessary for the proof of \cref{lem:lex cc from tcl gives shelling}. 
\begin{remark}\label{prop:labels of out and back roots}
    Let $P$ be a finite bounded poset with TCL-labeling $\lambda$ taking values in the integers. Fix a total order $\Gamma: m_1,\dots, m_t$ on the maximal chains of $P$ which is consistent with the lexicographic order on maximal chains induced by $\lambda$. If $r$ is a saturated chain containing $\zh$ which is contained in $m_i$ and $m_k$ with $i<k$, then for each $j$ with $i\leq j\leq k$, $\lambda(r)$ is an initial section of the label sequence $\lambda(m_j)$. 
   
\end{remark}

\begin{proposition}\label{prop:cc relabel of descent} Let $P$ be a finite bounded poset with TCL-labeling $\lambda$ taking values in the integers. Suppose that maximal chain $m$ has a topological descent at $x\lessdot y \lessdot z$ and that $c$ given by $x\lessdot x_1 \lessdot x_2 \lessdot \dots \lessdot z$ is the topologically ascending maximal chain of $[x,z]_{m^x}$. Then either $\lambda(m^x,x,x_1)<\lambda(m^x,x,y)$ or $\lambda(m^x,x,x_1)=\lambda(m^x,x,y)$ and $\lambda(m^x\cup x_1,x_1,x_2)<\lambda(m^x\cup y,y,z)$.
\end{proposition}

\begin{proof}
    This follows because the label sequence associated to $c$ with respect to $m^x$ 
    strictly precedes lexicographically the label sequence associated to $x\lessdot y \lessdot z$ with respect to $m^x$ 
\end{proof}

\begin{proposition}\label{prop: different lambda lable different lambda prime label}
    Let $P$ be a finite, bounded poset with a TCL-labeling $\lambda$. By \cref{rmk:lin ext of label set preserves tcl}, we may take $\lambda$ to be a labeling by integers. Fix a total order $\Gamma: m_1,\dots, m_t$ on the maximal chains of $P$ that agrees with the lexicographic order on maximal chains induced by $\lambda$. Let $\lambda'$ be the CE-labeling of $P$ determined by $\Gamma$ as in \cref{def:cc label from tcl}. Suppose $x\lessdot u\lessdot b$ and $x\lessdot y \lessdot d$ with $u\neq y$ and let $r$ be a root of $x$. If $\lambda(r,x,u)<\lambda(r,x,y)$, then $\lambda'(r,x,u)<\lambda'(r,x,y)$. Also, if 
    $\lambda(r,x,u)=\lambda(r,x,y)$
    and $\lambda(r\cup u,u,b)<\lambda(r\cup y,y,d)$, then $\lambda'(r\cup u,u,b)<\lambda'(r\cup y,y,d)$. 
\end{proposition}

\begin{proof}
     Let $a_1,\dots, a_s$ be the atoms of $[x,\oneh]_r$ in the order in which they first appear with $r$ in a maximal chain in $\Gamma$. 

    Suppose that $\lambda(r,x,u)<\lambda(r,x,y)$. 
    Let $a_j$ be the earliest atom such that $\lambda'(r,x,a_j)=\lambda'(r,x,u)$. Notice that by \cref{def:cc label from tcl}, $\lambda'(r,x,a_j)=i_j$ where $m_{i_j}$ is the first maximal chain in $\Gamma$ containing $r\cup a_j$. Otherwise, there would be some earlier atom than $a_j$, say $a$, such that $\lambda'(r,x,a)=\lambda'(r,x,u)$. Similarly, let $a_k$ be the earliest atom such that $\lambda'(r,x,a_k)=\lambda'(r,x,y)$. So,  $\lambda'(r,x,a_k)=i_k$ where $m_{i_k}$ is the first maximal chain in $\Gamma$ containing $r\cup a_k$. By \cref{prop:same cc relable label seqs} (i) and \cref{prop:labels of out and back roots}, $\lambda(r,x,a_j)=\lambda(r,x,u)$ and $\lambda(r,x,a_k)=\lambda(r,x,y)$. Since $\lambda(r, x, u) < \lambda(r, x, y)$, we have $\lambda(r,x,a_j)< \lambda(r,x,a_k)$ and thus $i_j<i_k$. Thus, $\lambda'(r,x,u)<\lambda'(r,x,y)$. 

    Now suppose that $\lambda(r,x,u)=\lambda(r,x,y)$ 
    and $\lambda(r\cup u,u,b)<\lambda(r\cup y,y,d)$. 
    Notice that $r\cup u\neq r\cup y$, but $\lambda(r\cup u)= \lambda(r\cup y)$, so $\lambda(r\cup u\cup b)$ strictly precedes $\lambda(r\cup y\cup d)$ lexicographically. Let $v_l$ be the earliest atom of $[u,\oneh]_{r\cup u}$ such that $\lambda'(r\cup u,u,v_l) = \lambda'(r\cup u,u,b)$. Notice that by \cref{def:cc label from tcl}, $\lambda'(r\cup u,u,v_l)=i_l$ where $m_{i_l}$ is the first maximal chain in $\Gamma$ containing $r\cup u\cup v_l$.
    Similarly, let $w_n$ be the earliest atom of $[y,\oneh]_{r\cup y}$ such that $\lambda'(r\cup y,y,w_n) = \lambda'(r\cup y,y,d) = i_n$ where $m_{i_n}$ is the first maximal chain in $\Gamma$ containing $r\cup y\cup w_n$. 
    Again by \cref{prop:same cc relable label seqs} (i) and \cref{prop:labels of out and back roots}, we have $\lambda(r\cup u,u,v_l) = \lambda(r\cup u,u,b)$ and $\lambda(r\cup y,y,w_n) = \lambda(r\cup y,y,d)$. Since $\lambda(r\cup u,u,b)<\lambda(r\cup y,y,d)$, we have $\lambda(r\cup u,u,v_l)<\lambda(r\cup y,y,w_n)$ and thus $i_l<i_n$. Thus, $\lambda'(r\cup u,u,b)<\lambda'(r\cup y,y,d)$.
\end{proof}

We are now ready to prove our main result of this section, \cref{lem:lex cc from tcl gives shelling}, which yields \cref{thm: tcl iff cc} when combined with Theorem 5.8 in \cite{HerStad}. 

\begin{lemma}\label{lem:lex cc from tcl gives shelling}
     Let $P$ be a finite bounded poset with a TCL-labeling $\lambda$. By \cref{rmk:lin ext of label set preserves tcl}, we may take $\lambda$ to be a labeling by integers. Fix a total order $\Gamma: m_1,\dots, m_t$ on the maximal chains of $P$ that agrees with the lexicographic order on maximal chains induced by $\lambda$. Let $\lambda'$ be the CE-labeling of $P$ determined by $\Gamma$ as in \cref{def:cc label from tcl}. Then $\lambda'$ is a CC-labeling. 
\end{lemma}

\begin{proof}
    By \cref{prop:same cc relable label seqs} (ii), in each rooted interval, $\lambda'$ gives distinct label sequences to distinct maximal chains and no label sequence of a maximal chain is a prefix of the label sequence of any other maximal chain. Thus, it suffices to show that $\lambda'$ preserves the topological descents of $\lambda$ (see Remark \ref{rmk: descent sets}). Suppose that the maximal chain $m$ has a topological descent at $x\lessdot y \lessdot z$ with respect to $\lambda$. Let $r=m^x$ be the chain that agrees with $m$ up to $x$. As in \cref{def:cc label from tcl}, take the atoms $a_1,\dots,a_s$ of $[x,\oneh]_r$ in the order in which they first appear with $r$ in a maximal chain in $\Gamma$. So, $y=a_l$ for some $l$. Let $c: x\lessdot a_k \lessdot w \lessdot \dots \lessdot z$ be the topologically ascending maximal chain of $[x,z]_{r}$ with respect to $\lambda$. By \cref{prop:cc relabel of descent}, either \begin{enumerate}
        \item [(a)] $\lambda(r,x,a_k)<\lambda(r,x,a_l)$ or

        \item [(b)] $\lambda(r,x,a_k)=\lambda(r,x,a_l)$ and $\lambda(r\cup a_k,a_k,w)<\lambda(r\cup a_l,a_l,z)$.
    \end{enumerate}
    If (a) holds, \cref{prop: different lambda lable different lambda prime label} implies $\lambda'(r, x, a_k) < \lambda'(r, x, a_l$), so $m$ has a topological descent at $x\lessdot a_l\lessdot z$ with respect to $\lambda'$. 
    
    If (b) holds, we will show that either we have $\lambda'(r,x,a_k)< \lambda'(r,x,a_l)$ or both $\lambda'(r,x,a_k) = \lambda'(r,x,a_l)$ and $\lambda'(r\cup a_k,a_k,w)<\lambda'(r\cup a_l,a_l,z)$. In either case, $x \lessdot a_l \lessdot z$ is a topological descent with respect to $\lambda'$.

    Suppose that (b) holds. We first establish some notation. Let $a_{l'}$ be the earliest atom of $[x, \hat{1}]_r$ such that $\lambda'(r,x,a_{l'}) = \lambda'(r,x,a_l)$ and let $a_{k'}$ be the earliest atom of $[x, \hat{1}]_r$ such that $\lambda'(r,x,a_{k'}) = \lambda'(r,x,a_k)$. Let $m_{i_{l'}}$ be the first chain in $\Gamma$ containing $r\cup a_{l'}$ and let $m_{i_{k'}}$ be the first chain in $\Gamma$ containing $r\cup a_{k'}$. Thus, we have $\lambda'(r,x,a_{k'}) = \lambda'(r,x,a_k) = i_{k'}$ and $\lambda'(r,x,a_{l'}) = \lambda'(r,x,a_l) = i_{l'}$. We also necessarily have $k'\leq k$ and $l'\leq l$. We are now ready to show that $\lambda'(r, x, a_k) \leq \lambda'(r, x, a_l)$.

    If $k'\leq l'$, then $i_{k'}\leq i_{l'}$ so $\lambda'(r,x,a_k)\leq \lambda'(r,x,a_l)$. If $l'<k'$, we must consider the two cases that either $l<k$ or $k<l$. Suppose $k<l$. Observe that $r\cup c$ followed by the portion of $m$ above $z$ is some maximal chain containing $r\cup a_k$ which precedes $m$, so $m_{i_k}$, the first maximal chain containing $r\cup a_k$, precedes $m$. Since $l<k$, $m_{i_l}$ precedes $m_{i_k}$ and $m_{i_k}$ precedes $m$. Since both $m_{i_l}$ and $m$ contain $r\cup a_l$ while $m_{i_k}$ is the first maximal chain containing $r\cup a_k$, \cref{def:cc label from tcl} gives $\lambda'(r,x,a_k) = \lambda'(r,x,a_l)$. If $k < l$, then knowing that $\lambda'(r, x, a_{l'})=\lambda'(r, x, a_l)$ and $l' < k' \leq k < l$ allows us to use \cref{prop:same cc relable label seqs} (i) to conclude that $\lambda'(r,x,a_k)= \lambda'(r,x,a_l)$.

    If $\lambda'(r, x, a_{k}) < \lambda'(r, x, a_{l})$, then $x \lessdot a_l \lessdot z$ is a topological descent in $m$ with respect to $\lambda'$ and we are done. It remains to consider the case where $\lambda'(r, x, a_k)=\lambda'(r, x, a_l)$. Recall that we are in case (b), that is, $\lambda(r, x, a_k)=\lambda(r, x, a_l)$ and 
    $\lambda(r\cup a_k,a_k,w)<\lambda(r\cup a_l,a_l,z)$. 
    So, by \cref{prop: different lambda lable different lambda prime label}, $\lambda'(r\cup a_k,a_k,w)<\lambda'(r\cup a_l,a_l,z)$. Thus, $x\lessdot a_l\lessdot z$ is a topological decent in $m$ with respect to $\lambda'$.  
    
\end{proof}

\begin{remark}\label{rmk:cc and tcl are the same}
    Notice that the proof of \cref{lem:lex cc from tcl gives shelling} shows that for a given TCL-labeling of a bounded poset $P$, the corresponding CC-labeling defined by \cref{def:cc label from tcl} has topological descents in exactly the same locations as the TCL-labeling. 
\end{remark} 
    
\begin{theorem}
    A finite bounded poset $P$ is TCL-shellable if and only if $P$ is CC-shellable.
    \label{thm: tcl iff cc}
\end{theorem}

\begin{proof}
    The forward direction is \cref{lem:lex cc from tcl gives shelling}. The backward direction is Theorem 5.8 in \cite{HerStad}. 
\end{proof}

\section{Recursive First Atom Sets}\label{sec:RFAS}

In the spirit of recursive atom orderings and generalized recursive atom orderings, we now introduce recursive first atom sets (RFAS) as a new method for producing shelling orders of finite bounded posets. Posets exist that admit an RFAS but that do not admit an RAO or GRAO (see Figure \ref{fig: ccnotclex} and Theorem \ref{thm: tcl iff RFAS}). These RFASs have the benefit of only requiring us to specify a first atom for each rooted interval, as opposed to requiring us to specify total orderings on atoms of upper rooted intervals.  

As the following definition is somewhat technical, the reader might find it useful to refer to Figures \ref{fig: why we need def} and \ref{fig:RFAS condition ii} and Examples \ref{ex: why we need def} and \ref{ex:not shellable but RFASi} as they read.

\begin{definition} \label{def:RFAS}
    The bounded poset $P$ admits a \textbf{recursive first atom set} (\textbf{RFAS}) if the length of $P$ is 1 or if for each $x < y$ in $P$ and each root $r$ for $[x, y]$, there exists an atom of $[x, y]$, denoted $\Omega(r, x, y)$ and called the \textbf{first atom} of $[x, y]_r$, satisfying the following whenever $a$ is an atom of $[x,y]_r$ and $b$ is the first atom of $[a,y]_{r\cup a}$:
    \begin{enumerate}
        \item [(i)] $a$ is the first atom in $[x,y]_r$ if and only if $a$ is the first atom in $[x,b]_r$
        \item [(ii)] 
        if $a$ is not the first atom in $[x,y]_r$ and $a'$ is the first atom in $[x,y]_r$, then there exist atoms $a_1, \dots, a_p$ of $[x,y]$ and elements $b_1,\dots, b_{p-1}$ such that:
        \begin{enumerate} 
            \item $a_p$ is the first atom of $[x,b]_r$, 
            \item $a_1=a'$, and 
            \item $b_i$ is the first atom of $[a_{i+1},y]_{r\cup a_i}$ and $a_{i}$ is the first atom of $[x,b_i]_r$ for all $i$ with $1 \leq i\leq p-1$
        \end{enumerate}
    \end{enumerate}

\end{definition}

\begin{figure}

\begin{tikzpicture}[scale=1.25]
\node at (-1.7,1.5) {$P=$};
\node [below] at (0,0) {$x$}; 
\draw [fill] (0,0) circle [radius=0.05]; 
\node [left] at (-1,1) { $a'$}; 
\draw [fill] (-1,1) circle [radius=0.05]; 
\node [right] at (1,1) { $a$}; 
\draw [fill] (1,1) circle [radius=0.05]; 
\node [left] at (-1,2) { $b'$}; 
\draw [fill] (-1,2) circle [radius=0.05]; 
\node [right] at (1,2) { $b$}; 
\draw [fill] (1,2) circle [radius=0.05]; 
\node [above] at (0,3) { $y$}; 
\draw [fill] (0,3) circle [radius=0.05];
\draw (1,1)--(-1,2);
\draw (0,0)--(-1,1)--(-1,2)--(0,3);
\draw (0,0)--(1,1)--(1,2)--(0,3);
\draw (-1,1)--(1,2);
\draw (1,1)--(-1,2);
\end{tikzpicture}
 \hskip 1.3cm
  \begin{tikzpicture}[scale=1.25]
  \node at (-1.7,1.5) {$Q=$};
\node [below] at (0,0) {$x$}; 
\draw [fill] (0,0) circle [radius=0.05]; 
\node [left] at (-1,1) { $a'$}; 
\draw [fill] (-1,1) circle [radius=0.05]; 
\node [right] at (1,1) { $a$}; 
\draw [fill] (1,1) circle [radius=0.05]; 
\node [left] at (-1,2) { $b'$}; 
\draw [fill] (-1,2) circle [radius=0.05]; 
\node [right] at (1,2) { $b$}; 
\draw [fill] (1,2) circle [radius=0.05]; 
\node [above] at (0,3) { $y$}; 
\draw [fill] (0,3) circle [radius=0.05];
\node [right] at (2,1) { $a''$}; 
\draw [fill] (2,1) circle [radius=0.05];
\node [right] at (2,2) { $b''$}; 
\draw [fill] (2,2) circle [radius=0.05];

\draw (1,1)--(-1,2);
\draw (0,0)--(-1,1)--(-1,2)--(0,3);
\draw (0,0)--(1,1)--(1,2)--(0,3);
\draw (-1,1)--(1,2);
\draw (1,1)--(-1,2);
\draw (0,0)--(2,1)--(2,2)--(0,3);
\end{tikzpicture}

\caption{The posets $P$ (left) and $Q$ (right) discussed in Example \ref{ex: why we need def} and Example \ref{ex:not shellable but RFASi}, respectively.}
\label{fig: why we need def}
\end{figure}
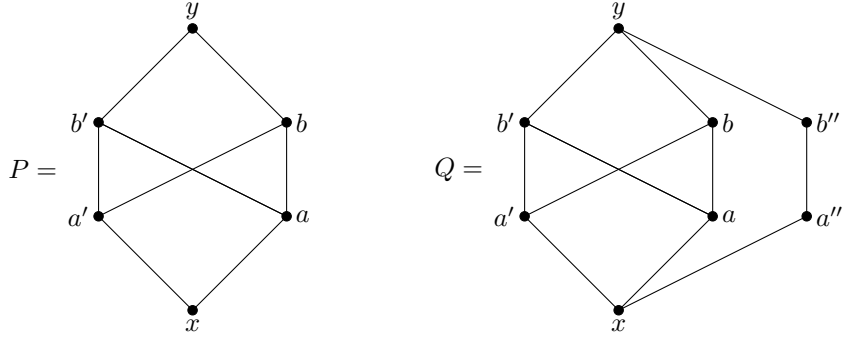

\begin{figure}
\begin{center}
\begin{tikzpicture}[scale=0.75]

\node at (6,1){${\zh}$};
\node at (6,3){${u}$};

    \node at (0,5){${a'=a_1}$};
    \node at (2,5){$a_2$};
    \node at (4,5){$a_3$};
    \node at (6,5){$\dots$};
    \node at (8,5){$a_{p-1}$};
    \node at (10,5){$a_p$};
    \node at (12,5){${a}$};


\node at (0,7){$b_1$};
\node at (2,7){$b_2$};
\node at (4,7){$b_3$};

\node at (8,7){$b_{p-2}$};

\node at (10,7){$b_{p-1}$};

\node at (12,7){$b$};

\node at (6,11){${y}$};

\draw [dashed] (6,1.3)--node [left] {$\mathbf{r}$}(6,2.8);

\draw (5.8,3.1)--(0.2,4.8);
\draw (6.2,3.1)--(11.8,4.8);
\draw (5.9,3.2)--(2.1,4.8);
\draw (5.9,3.2)--(4.1,4.8);
\draw (6.1,3.2)--(7.9,4.8);
\draw (6.1,3.2)--(9.9,4.8);

\draw (12,5.2)--(12,6.8);
\draw (10.2,5.2)--(11.8,6.8);
\draw (10,5.2)--(10,6.8);
\draw (8.2,5.2)--(9.8,6.7);
\draw (8,5.2)--(8,6.8);
\draw (6.5,5.2)--(7.8,6.7);
\draw (5.5,5.2)--(4.2,6.8);
\draw (4,5.2)--(4,6.7);
\draw (3.8,5.2)--(2.2,6.7);
\draw (2,5.2)--(2,6.7);
\draw (1.8,5.2)--(0.2,6.7);
\draw (0,5.2)--(0,6.7);

\draw [dashed] (0,7.2)--(5.8,10.8);
\draw [dashed] (2.1,7.2)--(5.9,10.8);
\draw [dashed] (4,7.2)--(6,10.8);
\draw [dashed] (8,7.2)--(6,10.8);
\draw [dashed] (10,7.2)--(6.1,10.8);
\draw [dashed] (12,7.2)--(6.2,10.8);

\end{tikzpicture}

\caption{A diagram illustrating condition (ii) of \cref{def:RFAS}}
\label{fig:RFAS condition ii}
\end{center}
\end{figure}
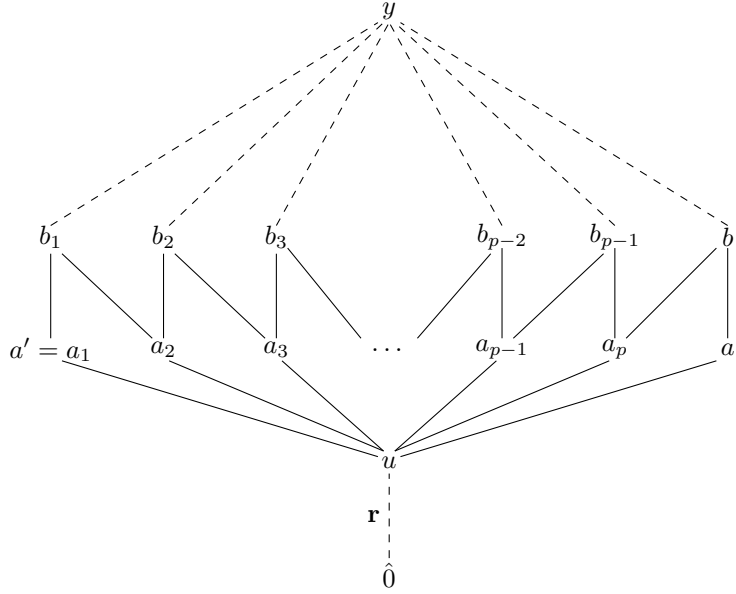

The goal of \cref{def:RFAS} is to specify an atom in every rooted interval of the bounded poset $P$ in such a way that (1) a total order on the maximal chains of $P$ is completely determined by these specified atoms, and (2) this total order on the maximal chains is a shelling order. The sense in which we determine a total order on maximal chains from a specification of atoms is given by the following notion of consistency. Let $\mathcal{C}$ be such a specification of atoms. Say that a total order on maximal chains of $P$ is consistent with $\mathcal{C}$ if, for any rooted interval $[u, v]_r$ in $P$, the earliest maximal chain in the total order that contains $u,v$, and $r$ also contains the specified atom in the interval $[u, v]_r$. Condition (i) is necessary to guarantee that there exists a total order that is consistent with a given specification of atoms. Condition (ii), which is illustrated in \cref{fig:RFAS condition ii}, is necessary to ensure that the resulting total order is a shelling order. The following two examples highlight these points. 

\begin{example}\label{ex: why we need def}
Consider the poset $P$ shown in Figure \ref{fig: why we need def} and the collection $\mathcal{C}$ of first atoms for intervals in $P$ where $a$ is first in $[x, b]$, $a'$ is first in $[x, b']$ and $[x, y]$, and $b$ is first in $[a', y]$ and $[a, y]$. Note there is only one root for each of these intervals, so we are not specifying the root here, and we are not specifying first atoms in intervals containing only one atom. Notice that $\mathcal{C}$ does not satisfy the backwards direction of Definition \ref{def:RFAS} (i). Since $a'$ is first in $[x,y]$ and $b$ is first in $[a',y]$, any total order on maximal chains that is consistent with $\mathcal{C}$ must begin with the chain $m=\{x,a',b,y\}$. Notice that $m$ is then the earliest chain containing $x$ and $b$. However, $m$ does not contain $a$, the first atom in $[x,b]$. So, any such total order is not consistent with $\mathcal{C}$.

Likewise, consider the collection $\mathcal{C'}$ of first atoms where  $a'$ is first in $[x, b]$, $a$ is first in $[x, b']$ and $[x, y]$, and $b$ is first in $[a', y]$ and $[a, y]$.  Notice $\mathcal{C'}$ does not satisfy the forward direction of Definition \ref{def:RFAS} (i). Since $a$ is first in $[x,y]$ and $b$ is first in $[a,y]$, any total order on maximal chains that is consistent with $\mathcal{C'}$ must begin with the chain $m'=\{x,a,b,y\}$. Notice that $m'$ is then the earliest chain containing $x$ and $b$. However, $m'$ does not contain $a'$, the first atom in $[x,b]$. So, any such total order is not consistent with $\mathcal{C'}$.

\end{example}

\begin{example}\label{ex:not shellable but RFASi}
    For the poset $Q$ in \cref{fig: why we need def}, define a first atom set as follows: $a''$ is first in $[x,b'']$ and $[x,y]$, $a$ is first in $[x,b]$, $a'$ is first in $[x,b']$, $b''$ is first in $[a'',y]$, $b'$ is first in $[a,y]$, and $b$ is first in $[a',y]$. Just as in \cref{ex: why we need def}, we need not specify roots for any of these intervals. It can easily be checked that this first atom set satisfies condition (i) but not condition (ii) of \cref{def:RFAS}. Further, $Q$ is clearly not shellable. 
    \end{example}

\begin{remark}\label{rmk: restriction of rfas is rfas}
    Observe that if $P$ admits an RFAS $\Omega$, then restricting $\Omega$ to any closed rooted interval $[x,y]_r$ of $P$ clearly gives an RFAS of $[x,y]$.
\end{remark}

The following several definitions and propositions are used to prove \cref{thm:RFAS implies shellable} which says that an RFAS produces a shelling order of a finite bounded poset. While these intermediate steps are quite technical, in essence they are identifying structures from an RFAS that behave like ascents and descents (or topological ascents and topological descents) in a lexicographic shelling. First we note that although we have shown through \cref{thm:cc but not cl posets} that topological ascents and topological descents are structurally different than ascents and descents, these two different notions of ascents and descents function in exactly the same way in terms of proving that the lexicographic orders from the different labeling types are shelling orders. It is for this purpose of producing shellings that we want to draw parallels between structures in an RFAS and ascents and descents, so we only use the words ``ascent" and ``descent" throughout the rest of this section, even though we could correctly say ``topological ascent" and ``topological descent" as well.

For a finite bounded poset with an RFAS, the following definition identifies a chain in each rooted interval that plays a role similar to that of an ascending chain in a lexicographic shelling.

\begin{definition}\label{def: first atom chain from RFAS}
    Let $P$ be a finite, bounded poset and let $\Omega$ be an RFAS of $P$. For any rooted interval $[x,y]_r$, define the \textbf{first atom chain of $[x,y]_r$}, denoted $c(r,x,y)$, to be the unique maximal chain $c(r,x,y)$ of $[x,y]_r$ given by $x=z_0\lessdot z_1\lessdot \dots \lessdot z_q=y$ satisfying $z_{i+1}$ is the first atom of $[z_i,y]_{r\cup\{z_0,\dots,z_i\}}$ according to $\Omega$ for all $i$ such that $0\leq i \leq q-2$.
\end{definition}

Now we use an RFAS to identify chains that function the same way that descents do in a lexicographic shelling.

\begin{definition}\label{def:pseudo descent}
    Let $P$ be a finite, bounded poset and let $\Omega$ be an RFAS of $P$. For a maximal chain $m$ of $P$ given by $\zh=x_0\lessdot x_1\lessdot \dots \lessdot x_n=\oneh$, say that $x_i\lessdot x_{i+1}\lessdot x_{i+2}$ is a \textbf{pseudo descent} of $m$ if $x_{i+1}$ is not the first atom of $[x_i,x_{i+2}]_{m^{x_i}}$ according to $\Omega$. 
\end{definition} 

The following proposition shows that for an RFAS, any chain that does not behave like an ascending chain must contain a chain that behaves like a descent.

\begin{proposition}\label{prop: not first atom then have pseudo descent}
    Let $P$ be a finite, bounded poset and let $\Omega$ be an RFAS of $P$. Let $[x,y]_r$ be a rooted interval of $P$ and let $m$ be a maximal chain of $[x,y]_r$. If $m\neq c(r,x,y)$, then $m$ contains a pseudo descent. 
\end{proposition}

\begin{proof}
    Say that $m$ is given by $x=x_0\lessdot x_1 \lessdot \dots \lessdot x_{p}=y$ and $m\neq c(r,x,y)$. 
    We prove that $m$ contains a pseudo descent by induction on the length of of the longest chain in $[x,y]$. It holds vacuously if the longest chain in $P$ is length 1. Suppose $[x,y]$ has some chain longer than length 1. Since $m\neq c(r,x,y)$, there exists a unique $k$ with $0\leq k\leq p-2$ such that $x_k \in c(r,x,y)$ and $c(r,x,y)^{x_k}=m^{x_k}$ while $c(r,x,y)$ and $m$ contain different atoms of $[x_k,\oneh]$. By definition, $c(r,x,y)$ contains $\Omega(m^{x_k},x_k,\oneh)$ 
    the first atom of $[x_k,\oneh]_{m^{x_k}}$, while $m$ contains $x_{k+1}\neq \Omega(m^{x_k},x_k,\oneh)$. So, $x_{k+1}$ is not the first atom of $[x_k,\oneh]_{m^{x_k}}$ according to $\Omega$. If $x_{k+2}$ is not the first atom of $[x_{k+1},\oneh]_{m^{x_{k+1}}}$ according to $\Omega$, then above $x_{k+1}$, $m$ is not $c(r\cup\{x_1,\dots,x_{k+1}\},x_{k+1},\oneh)$. Thus by induction, $m$ has a pseudo descent above $x_{k+1}$. If $x_{k+2}$ is the first atom of $[x_{k+1},\oneh]_{m^{x_{k+1}}}$ according to $\Omega$, then by \cref{def:RFAS} (i) the first atom $\Omega(m^{x_k},x_k,x_{k+2})$ of $[x_k,x_{k+2}]_{m^{x_k}}$ is not $x_{k+1}$. In this case, $x_k\lessdot x_{k+1}\lessdot x_{k+2}$ is a pseudo descent of $m$.
\end{proof}

On the other hand, the next proposition shows that in an RFAS, the chains that function as ascending chains do not contain any of the chains that behave like descents.

\begin{proposition}\label{prop:RFAS first atom chain no pseudo decents}
    Let $P$ be a finite, bounded poset and let $\Omega$ be an RFAS of $P$. For any rooted interval $[x,y]_r$, the first atom chain $c(r,x,y)$ does not contain a pseudo descent. 
\end{proposition}

\begin{proof}
    This follows directly from the forward direction of condition (i) in \cref{def:RFAS}.
\end{proof}

 Now we use these ascent-like (first atom chains) and descent-like (pseudo descents) structures to construct a total order on the maximal chains of a finite bounded poset $P$ with an RFAS $\Omega$. In particular, in the following definition and the next three propositions, we use the first atom chains and pseudo descents to define a partial order on the maximal chains of the poset, denoted $\mathcal{M}_{\Omega}$. Any linear extension of $\mathcal{M}_{\Omega}$ will be shown to give a shelling order in \cref{thm:RFAS implies shellable}. Note that this partial order  
 is in essence a maximal chain descent order (MCDO) as introduced in \cite{LacinaMCDO}. \cref{thm:RFAS implies shellable} shows that $\mathcal{M}_{\Omega}$ encodes shellings of $P$ the same way that an MCDO does.

\begin{definition}\label{def:max chain order from RFAS}
    Let $P$ be a finite, bounded poset with an RFAS $\Omega$. For two maximal chains $m$ and $m'$ of $P$, define a relation $m\to m'$ exactly when there exist $x\lessdot y\lessdot z$ such that $x$ and $z$ are both in $m$ and $m'$, $m$ and $m'$ agree below $x$ and above $z$, $x\lessdot y \lessdot z$ is a pseudo descent of $m'$, and $m$ between $x$ and $z$ is the first atom chain $c(m^x,x,z)$. Denote by $\preceq$ the reflexive and transitive closure of all relations of the form $m\to m'$. This relation together with the set of maximal chains of $P$ will be denoted $\mathcal{M}_{\Omega}$. 
\end{definition}

We will show $M_{\Omega}$ is a partial order in \cref{prop:compatible max chain order from RFAS}. The next two propositions are technical results necessary for the proof of \cref{prop:compatible max chain order from RFAS}.

\begin{proposition}\label{prop:exists path to first atom chain}
    Let $P$ be a finite, bounded poset with an RFAS $\Omega$. Suppose $m$ is a maximal chain of $P$ that is not the first atom chain $c(\{\zh\},\zh,\oneh)$. Then there exists a sequence of maximal chains $c(\{\zh\},\zh,\oneh) \to m_1\to m_2\to \dots \to m$.
\end{proposition} 

\begin{proof}
Let $m$ be a maximal chain of $P$ given by $\zh=x_0\lessdot x_1\lessdot \dots \lessdot x_s=\oneh$ such that $m$ is not the first atom chain $c(\{\zh\},\zh,\oneh)$. We show there is a sequence of maximal chains $c(\{\zh\},\zh,\oneh) \to m_1\to m_2\to \dots \to m$ by induction on the length of the longest chain in $P$. If the length of the longest chain in $P$ is 1, this is vacuously true. If the length of the longest chain in $P$ is 2, this is clear by \cref{prop: not first atom then have pseudo descent}. Suppose that for some $n\geq 3$, whenever the length of the longest chain of $P$ is at most $n-1$, there is a sequence of maximal chains $c(\{\zh\},\zh,\oneh) \to m_1\to m_2\to \dots \to m$. 

Say the length of the longest maximal chain in $P$ is $n$. Since the longest maximal chain of $[x_1,\oneh]$ has length less than $n$, by induction there is a sequence of maximal chains $\zh \cup c(\{\zh,x_1\},x_1,\oneh)\to m_1'\to m_2' \to \dots \to m$. Now there are two cases: either $x_1$ is the first atom in $[\zh,\oneh]$ according to $\Omega$ or not. Suppose $x_1$ is the first atom in $[\zh,\oneh]$ according to $\Omega$. Then $\zh \cup c(\{\zh,x_1\},x_1,\oneh) = c(\{\zh\},\zh,\oneh)$, so we have the desired sequence of maximal chains. 

Suppose that $x_1$ is not the first atom in $[\zh,\oneh]$ according to $\Omega$. Let $b$ be the first atom of $[x_1,\oneh]_{\{\zh,x_1\}}$, that is, $b$ is the element of $c(\{\zh,x_1\},x_1,\oneh)$ covering $x_1$. Then by \cref{def:RFAS} (ii), there exist atoms $a_1,\dots, a_p$ of $P$ and elements $b_1,\dots, b_{p-1}$ such that $a_p$ is the first atom of $[x,b]_r$, $a_1$ is the first atom of $[\zh,\oneh]$ according to $\Omega$, $b_i$ is the first atom of $[a_{i+1},y]_{r\cup a_i}$, and $a_{i}$ is the first atom of $[x,b_i]_r$ for all $i$ with $1 \leq i\leq p-1$. Then we have $c(\{\zh\},\zh,b)\cup c(\{\zh,x_1\},x_1,\oneh)_b  
\to \zh \cup c(\{\zh,x_1\},x_1,\oneh)$ and $c(\{\zh\},\zh,b)$ contains $a_p$. Then again by induction we have a sequence of maximal chains $\zh\cup c(\{\zh,a_p\},a_p,\oneh)\to \tilde{m_1}\to \tilde{m_2}\to \dots \to c(\{\zh\},\zh,b)\cup c(\{\zh,x_1\},x_1,\oneh)_b 
$. Note that $\zh\cup c(\{\zh,a_p\},a_p,\oneh)$ contains $b_{p-1}$, so we can repeat this process with $b_{p-1}$ and $a_{p-1}$ and so on until reaching $\zh \cup c(\{\zh,a_1\},a_1,\oneh)$ which is the first atom chain $c(\{\zh\},\zh,\oneh)$. Together, this gives a sequence of maximal chains $c(\{\zh\},\zh,\oneh) \to m_1\to m_2\to \dots \to m$. This shows that there exists such a sequence which begins with $c(\{\zh\},\zh,\oneh)$.
\end{proof}

\begin{proposition}\label{prop:all paths lead to first atom chain}
    Let $P$ be a finite, bounded poset with an RFAS $\Omega$. Then any sequence of maximal chains obtained by replacing pseudo descents in maximal chains of $P$ with the corresponding first atom chains must eventually reach $c(\{\zh\},\zh,\oneh)$. 
\end{proposition}

\begin{proof}
First, note that every maximal chain $m$ has some sequence of maximal chains $c(\{\zh\},\zh,\oneh) \to m_1\to m_2\to \dots \to m_k = m$ by \cref{prop:exists path to first atom chain}. Let $k$ be maximal among such sequences. We will show the desired result by induction on $k$, that is, on the length of the longest such sequence. We first show the base case of $k=1$. We do this by showing that if $c(\{\zh\},\zh,\oneh) \to m$ is the longest such sequence, then $m$ has exactly one pseudo descent. Suppose $x\lessdot y\lessdot z$ is the pseudo descent of $m$ which gives $c(\{\zh\},\zh,\oneh) \to m$, so $y$ is not in $c(\{\zh\},\zh,\oneh)$. Thus, we have $m^x=c(\{\zh\},\zh,x)$. Then by 
\cref{prop:RFAS first atom chain no pseudo decents}, $m^x$ has no pseudo descents below $x$. Assume seeking contradiction that $m$ has a pseudo descent $x'\lessdot y' \lessdot z'$ with $y\leq x'$. Let $m'$ be the maximal chain obtained from $m$ by replacing the pseudo descent above $y$ with the corresponding first atom chain. Then we have $m'\to m$ while $m'\neq c(\{\zh\},\zh,\oneh)$ since $m'$ contains $y$. By \cref{prop:exists path to first atom chain}, there is a sequence of maximal chains $c(\{\zh\},\zh,\oneh)\to m_1'\to m_2'\to \dots \to m'$. But this contradicts that the longest sequence from $c(\{\zh\},\zh,\oneh)$ to $m$ has length one, since $c(\{\zh\},\zh,\oneh)\to m_1'\to m_2'\to \dots \to m' \to m$ has length at least two. Similarly, assuming that $m$ contains a pseudo descent $x'\lessdot x\lessdot y$, contradicts the fact that the longest sequence from $c(\{\zh\},\zh,\oneh)$ to $m$ has length one. Hence, $m$ has a single pseudo descent which means that $c(\{\zh\},\zh,\oneh)\to m$ is the only such sequence for $m$ proving the base case.

Now suppose that for some $n\geq 2$, whenever a maximal chain $m$ has $k<n$, any sequence of maximal chains obtained by beginning with $m$ and replacing pseudo descents with the corresponding first atom chains must eventually reach $c(\{\zh\},\zh,\oneh)$. Say $m$ is a maximal chain with $k=n$. Since $m\neq c(\{\zh\},\zh,\oneh)$, $m$ has at least one pseudo descent by \cref{prop: not first atom then have pseudo descent}. For any maximal chain $m'$ such that $m'\to m$, that is, for any maximal chain obtained from $m$ by replacing a pseudo descent with the corresponding first atom chain, the longest sequence $c(\{\zh\},\zh,\oneh)\to \tilde{m_1}\to \tilde{m_2}\to \dots \to m' $ is at most length $n-1$. Otherwise, there would be such a sequence from $c(\{\zh\},\zh,\oneh)$ to $m$ of length at least $n+1$. Then by the induction hypothesis, any sequence of maximal chains obtained by beginning with $m'$ and replacing pseudo descents with the corresponding first atom chains must eventually reach $c(\{\zh\},\zh,\oneh)$. Then since $m'$ was obtained from an arbitrary pseudo descent of $m$, any sequence of maximal chains obtained by beginning with $m$ and replacing pseudo descents with the corresponding first atom chains must eventually reach $c(\{\zh\},\zh,\oneh)$. 
\end{proof}

\begin{proposition}\label{prop:compatible max chain order from RFAS}
Let $P$ be a finite, bounded poset with an RFAS $\Omega$. Let $\mathcal{M}_{\Omega}$ be as defined in \cref{def:max chain order from RFAS}. Then $\mathcal{M}_{\Omega}$ is a partial order and has a unique minimal element given by the first atom chain $c(\{\zh\},\zh,\oneh)$.    
\end{proposition}

\begin{proof}
To show that $\preceq$ is a partial order, it suffices to show that $\preceq$ is antisymmetric. We prove this by showing that no sequence of maximal chains $m_1\to m_2\to \dots \to m$ of at least length one has $m_1=m$. Recall that $c(\{\zh\},\zh,\oneh)$ contains no pseudo descents by \cref{prop:RFAS first atom chain no pseudo decents}. Thus, $m_i\neq c(\{\zh\},\zh,\oneh)$ for any $i$ such that $1<i$. Assume that there is some sequence $m_1\to m_2\to \dots \to m$ of at least length one that has $m_1=m$. Then $c(\{\zh\},\zh,\oneh)$ is not one of the chains in this sequence and the sequence can be repeated forever without reaching $c(\{\zh\},\zh,\oneh)$. This contradicts \cref{prop:all paths lead to first atom chain}. Hence, $\preceq$ is antisymmetric and thus is a partial order. Moreover, \cref{prop:exists path to first atom chain} implies that $c(\{\zh\},\zh,\oneh) \preceq m$ for all maximal chains $m$ of $P$, so $c(\{\zh\},\zh,\oneh)$ is the unique minimal element with respect to $\preceq$.
\end{proof}

Now we come to the main result of this section, namely that a finite bounded poset that admits an RFAS is shellable. The proof of this theorem essentially follows the structure of Bj{\"o}rner and Wachs' proof of lexicographic shellability in \cite{non-pure1}.

\begin{theorem}\label{thm:RFAS implies shellable}
    Suppose finite, bounded poset $P$ admits an RFAS $\Omega$. Let $M_{\Omega}$ be as defined in \cref{def:max chain order from RFAS}. Then any linear extension of $M_{\Omega}$ gives a shelling of the order complex $\Delta(P)$.
\end{theorem}

\begin{proof}
    Let the total order on maximal chains $\Gamma:m_1,\dots, m_t$ be a linear extension of $\preceq$. Consider maximal chains $m_i$ and $m_j$ such that $i<j$. We will find a maximal chain $m_{i'}$ such that $i'<j$ and $m_i\cap m_j\subseteq m_{i'}\cap m_j$ with $m_{i'}\cap m_j=m_j\setminus\{y\}$ for some $y\in m_j$. This then shows that $\Gamma$ gives a shelling order of $\Delta(P)$. 
    
    We first identify all of the intervals in which $m_i$ and $m_j$ disagree. Say that $[u,v]$ is an interval where $m_i$ and $m_j$ differ when $u$ and $v$ are contained in both $m_i$ and $m_j$ but $m_i$ and $m_j$ share no common elements in $(u,v)$. Let $s$ be the number of intervals where $m_i$ and $m_j$ differ. 
    We show by induction on $s$ that $m_j$ contains a pseudo descent in at least one of the intervals where $m_i$ and $m_j$ differ. 
    Suppose $s=1$ and the one interval where $m_i$ and $m_j$ disagree is $[u,v]$. 
    Assume seeking contradiction that $m_j$ does not contain a pseudo descent in the interval $[u,v]$. Then $m_j$ restricted to $[u,v]$ is the first atom chain of $[u,v]_{m_j^{u}}$. Then by \cref{rmk: restriction of rfas is rfas} and \cref{prop:compatible max chain order from RFAS}, we have $m_j\prec m_i$. But this contradicts that $\Gamma$ is a linear extension of $\preceq$ since $m_i$ precedes $m_j$ in $\Gamma$. Thus, $m_j$ contains a pseudo descent in the interval $[u,v]$. Now suppose $s>1$. Let $[u,v]$ be the the first interval on which $m_i$ and $m_j$ differ in the sense that $m_i$ and $m_j$ agree everywhere below $u$. If $m_j$ contains a pseudo descent in $[u,v]$, then we are done. Otherwise $m_j$ restricted to $[u,v]$ is the first atom chain of $[u,v]_{m_j^{u}}$. Let $d'$ be the portion of $m_i$ above $v$. Then again by \cref{rmk: restriction of rfas is rfas} and \cref{prop:compatible max chain order from RFAS}, we have $m_j^{v}\cup d'\prec m_i$. This implies that $m_j^{v}\cup d'$ precedes $m_i$ in $\Gamma$, so $m_j^{v}\cup d'$ precedes $m_j$ in $\Gamma$. Further, $m_j^{v}\cup d'$ and $m_j$ disagree in only $s-1$ intervals and those intervals are among the intervals in which $m_i$ and $m_j$ differ.  
    Thus, by induction $m_j$ has a pseudo descent in an interval where $m_i$ and $m_j$ disagree.

    Now let $[u,v]$ be an interval where $m_i$ and $m_j$ differ such that $m_j$ has a pseudo descent in $[u,v]$. Let $x\lessdot y\lessdot z$ be this pseudo descent of $m_j$ in $[u,v]$. Let $c(m_j^{x},x,z)$ be the first atom chain of $[x,z]_{m_j^{x}}$ and let $d$ be the portion of $m_j$ above $z$. Then again by \cref{rmk: restriction of rfas is rfas} and \cref{prop:compatible max chain order from RFAS}, we have $m_j^{x} \cup c(m_j^{x},x,z) \cup d \prec m_j$. Thus, $m_j^{x} \cup c(m_j^{x},x,z) \cup d = m_{i'}$ for some $i'<j$ since $\Gamma$ is a linear extension of $\preceq$. Finally, observe that $m_i\cap m_j\subseteq m_{i'}\cap m_j$ with $m_{i'}\cap m_j=m_j\setminus\{y\}$ by construction. Therefore, $\Gamma$ gives a shelling of $\Delta(P)$.
\end{proof}

\begin{remark}
    The proof of \cref{thm:RFAS implies shellable} also shows that the restriction map of any of these linear extension shellings from an RFAS is given by the pseudo descents of the RFAS. This is the analog of the fact that in a lexicographic shelling, the restriction map is given by the (topological) descents.
\end{remark}

\section{Labeling-Compatible Recursive First Atom Sets (LCRFAS) and Equivalence to TCL-Shellability}\label{sec:lcrfas}

In this section, we add a condition to the definition of an RFAS, producing what we call labeling-compatible recursive first atom sets (LCRFAS) in \cref{def:LCRFAS}. In \cref{thm: tcl iff RFAS}, we show that posets have sets of atoms satisfying this definition if and only if they are TCL-shellable. We also provide in \cref{ex:RFAS not LCRFAS} a poset with an RFAS for which there is no labeling ``compatible" with the given RFAS. This shows that in \cref{thm: tcl iff RFAS}, an LCRFAS cannot be replaced with an RFAS.

\begin{definition}\label{def:LCRFAS}
    The finite bounded poset $P$ admits a \textbf{labeling compatible recursive first atom set (LCRFAS)} if $P$ admits an RFAS $\Omega$ (as in \cref{def:RFAS}) and $\Omega$ satisfies the following condition on the partial order $\preceq$ on the maximal chains of $P$ induced by $\Omega$ (as in \cref{def:max chain order from RFAS}). 
    \begin{itemize}
        \item [(LC)]  There exists a linear extension $\Gamma:m_1,\dots, m_t$ of $\preceq$ such that for maximal chains $m_i, m_j$, and $m_k$ with $i < j < k$, whenever $m_i$ and $m_k$ contain $r\cup\{x \lessdot y \lessdot z\}$, and $m_j$ contains $r \cup\{x \lessdot y'\}$, then $y'=y$  
        where $x, y, y', z$ are in $P$ and $r$ is a root of $x$ in $P$.
    \end{itemize}
    
\end{definition}

The following lemma is one direction of \cref{thm: tcl iff RFAS}, that is, a poset that is TCL-shellable admits an LCRFAS. The main idea is to let the first atom in a rooted interval be the atom contained in the topologically ascending chain of the rooted interval. This is similar in spirit to the main idea of Bj{\"o}rner and Wachs' proof that a poset is CL-shellable if and only if it admits an RAO.

\begin{lemma} \label{lem:TCL implies RFAS} 
If a finite, bounded poset $P$ admits a TCL-labeling, then $P$ admits an LCRFAS.
\end{lemma}
\begin{proof}

Suppose that the finite, bounded poset $P$ admits a TCL-labeling $\lambda$. By \cref{thm: tcl iff cc}, $P$ admits a CC-labeling $\lambda'$ as in \cref{def:cc label from tcl}. Note that $\lambda'$ takes values in the integers by definition. Define $\Omega$ as follows: for each rooted interval $[x,y]_r$, set $\Omega(r,x,y) = u$ where $u$ is the atom of $[x,y]$ contained in the unique topologically ascending chain of $[x,y]_r$ with respect to $\lambda'$. We show that $\Omega$ is an LCRFAS (satisfies \cref{def:LCRFAS}). 

Consider any atom $a$ of $[x,y]_r$. Suppose $\Omega(r\cup a,a,y)= b$, that is, $b$ is the atom of $[a,y]_{r\cup a}$ contained in the topologically ascending chain $c$ of $[a,y]_{r\cup a}$ with respect to $\lambda'$. First we show that \cref{def:RFAS} (i) holds by considering the following two cases: either $a$ is the first atom of $[x,y]_r$ or $a$ is not the first atom of $[x,y]_r$.

First suppose that $a$ is the first atom of $[x,y]_r$ according to $\Omega$. Since $b$ is the atom of $[a,y]_{r\cup u}$ contained in the topologically ascending chain $c$ of $[a,y]_{r\cup u}$ with respect to $\lambda'$ and $a$ is the atom contained in the topologically ascending chain of $[x,y]_{r}$ with respect to $\lambda'$, $a\cup c$ is the topologically ascending chain of $[x,y]_{r}$. Thus, $x\lessdot a \lessdot b$ is a topological ascent with respect to $r$ and $\lambda'$, so $a=\Omega(r,x,b)$.

Next suppose that $a$ is not the first atom of $[x,y]_r$ with respect $\Omega$. 
Observe that $x\cup c$ is not the topologically ascending chain of $[x,y]_r$ because $a$ is in $c$. Thus, $x\cup c$ must have a topological descent with respect to $\lambda'$. Since $c$ is topologically ascending in $[a,y]_{r\cup a}$, $x\lessdot a \lessdot b$ must be the topological descent in $x\cup c$ with respect to $r$ and 
$\lambda'$. Let $d$ be the topologically ascending chain of $[x,b]_r$ with respect to $\lambda'$ and let $u'$ be the atom of $[x,b]$ in $d$. Then $\Omega(r,x,b)=u'$ by definition of $\Omega$ and $u'\neq a$. 

Next we show that \cref{def:RFAS} (ii) holds. Let us suppose that $a$ is not the first atom of $[x,y]_r$ and $a'$ is the first atom of $[x,y]_r$.
Then, as previously observed, $x\lessdot a \lessdot b$ must be a topological descent with respect to $r$ and $\lambda'$. Let $a_p$ be the first atom of $[x,b]_r$ and $b_{p-1}$ be the first atom $[a_p,y]_{r\cup a_p}$. 
If $a_p\neq a'$, then we can repeat the same argument with $a_p$ and $b_{p-1}$ to produce $a_{p-1}$ and $b_{p-2}$ where $a_{p-1}$ is the first atom in $[x,b_{p-1}]_{r}$ and $b_{p-2}$ is the first atom in $[a_{p-1},y]_{r\cup a_{p-1}}$. If $a_{p-1}\neq a'$, then we repeat this process to produce $a_{p-2}$ and $b_{p-3}$, and so on. Notice that for any $i \leq p$, 
the lexicographically first chain in $[x,y]_r$ with respect to $\lambda'$ containing $a_{i-1}$ and $b_{i-2}$ comes strictly earlier lexicographically than the lexicographically first chain in $[x,y]_r$ containing $a_i$ and $b_{i-1}$. 
This process must eventually reach $a'$ since $a'$ is contained in the lexicographically first chain in $[x,y]_r$. Thus, \cref{def:RFAS} (ii) holds and $\Omega$ is an RFAS of $P$.

Lastly we show that condition (LC) of \cref{def:LCRFAS} holds so that $\Omega$ is an LCRFAS. Let $\preceq$ be the partial order on the maximal chains of $P$ induced by $\Omega$ as in \cref{def:max chain order from RFAS}. Observe that any total order $\Gamma:m_1,
\dots, m_t$ on the maximal chains of $P$ that is compatible with the lexicographic order induced by $\lambda'$ is a linear extension of $\preceq$. Assume seeking contradiction that condition (LC) of \cref{def:LCRFAS} does not hold. Then there are maximal chains $m_i, m_j, m_k$ in $\Gamma$ with $i < j < k$ where $m_i$ and $m_k$ contain $r \cup \{x \lessdot y \lessdot z\}$, while $m_j$ contains $r \cup \{x \lessdot y' \lessdot z'\}$ where $y \neq y'$. Then by \cref{prop:labels of out and back roots}, $\lambda'(r,x, y)=\lambda'(r,x, y')$ and $\lambda'(r\cup y, y, z)=\lambda'(r\cup y', y', z')$. However, this contradicts \cref{prop:same cc relable label seqs} (ii) since $y\neq y'$. Therefore, $\Omega$ is an LCRFAS.
\end{proof}

Now we precisely define what we mean when we say that a CE-labeling is ``compatible" with an RFAS. This definition and the following lemma are necessary for the opposite direction of \cref{thm: tcl iff RFAS}. We show in the subsequent, quite technical lemma that a poset that admits an LCRFAS also admits a CE-labeling that is compatible with the LCRFAS. It is then somewhat straightforward to prove that a CE-labeling that is compatible with an RFAS is a TCL-labeling, which we do in the proof of \cref{thm: tcl iff RFAS}.

\begin{definition} \label{def:compatible RFAS}
    Let $P$ be a finite, bounded poset. Let $\lambda$ be a CE-labeling on $P$ and let $\Omega$ be an RFAS for $P$. We say $\lambda$ is \textbf{compatible} with $\Omega$ if the following holds for all
    rooted intervals $[x,y]_r$: whenever $\Omega$ gives $a$ as the first atom of $[x, y]_r$, $a$ is contained in a maximal chain in $[x, y]_r$ that has the lexicographically smallest label sequence with respect to $\lambda$.  
\end{definition}

\begin{figure}
\begin{center}
\begin{tikzpicture}[scale=0.75]

\node at (5,1){${\zh}$};

\node at (5,3){${u}$};

    \node at (2,4){${v'}$};

    \node at (8,4){${v}$};

\node at (2,5.2){${w'}$};


\node at (2,7){${w}$};

\node at (4,7){${\hat{w}}$};

\node at (6,7){${\tilde{w}}$};

\node at (8,7){};
\draw [fill] (8,7) circle [radius=0.05];

\node at (5,10){${\oneh}$};

\draw [dashed] (5,1.3)--(5,2.8);

\draw (4.8,3.1)--(2.1,3.8);
\draw (5.2,3.1)--(7.9,3.8);

\draw (2,4.2)--(2,4.9);

\draw (7.8,4.2)--(2.2,6.8);
\draw (7.9,4.2)--(4.2,6.8);
\draw (8,4.2)--(6.1,6.8);
\draw (8,4.2)--(8,7);

\draw [dashed] (2,5.4)--(2,6.8);
\draw [dashed] (1.8,5.4) to [bend left=20] (1.2,7);
\draw [dashed] (1.2,7) to [bend left=20] node [align=center, minimum size=2cm, midway,  above=-2em, sloped] {\scalebox{1}{$\mathbf{c}(\mathbf{r}\cup \{{v'},{w'}\},{w'},{\hat{1}})$}} (4.8,9.8);

\draw [dashed] (2.1,7.2)--(4.9,9.7);

\draw [dashed] (4.1,7.2)--(5,9.7);

\draw[dashed] (6.2,7.2) to [bend right=20] (5,9.7);
\draw[dashed] (5.8,7.2) to [bend left=20] (5,9.7);

\draw [dashed] (8,7)--(5.2,9.8);

\node at (4.8,2) {$\mathbf{r}$};
\node at (3.1,4.6) {\scalebox{1}{$\mathbf{c}(\mathbf{r},u,w)$}};
\node at (1,5.7) {$\mathbf{m_j}$};
\node at (2.6,8) {$\mathbf{d}$};
\node at (3.1,6.6) {$\mathbf{m}$};
\node at (5.1,6.6) {$\mathbf{m'}$};
\node at (5.1,7.5) {$\mathbf{m_i}$};
\node at (6.6,7.5) {$\mathbf{m_k}$};
\node at (8.5,6.4) {$\mathbf{m_h}$};

\end{tikzpicture}

\caption{A schematic illustrating the argument in the proof of \cref{lem: RFAS has compatible labeling} in which $m_h,m_i,m_j,m',m_k,m$ is the relative order of the exhibited maximal chains. Chains are denoted with bold text, while elements are denoted with italic text.}
\label{fig:LCRFAS argument pic}
\end{center}
\end{figure}

\begin{lemma}\label{lem: RFAS has compatible labeling}
    Suppose finite, bounded poset $P$ admits an LCRFAS $\Omega$. Then $P$ admits a CE-labeling $\lambda$ such that $\lambda$ and $\Omega$ are compatible. 
\end{lemma}

\begin{proof}
    Let $\preceq$ be the partial order on the maximal chains induced by $\Omega$ as in \cref{def:max chain order from RFAS} and let $\Gamma$ be a linear extension of $\preceq$ satisfying (LC) of \cref{def:LCRFAS}. Let $\lambda$ be the CE-labeling of $P$ induced by $\Gamma$ as in \cref{def:cc label from tcl}. We will show that $\lambda$ is compatible with $\Omega$. 
    
Say that $\Gamma$ is given by $m_1,\dots, m_t$. We show that $\lambda$ is compatible with $\Omega$ by first showing that if, according to $\Omega$, $u\lessdot v\lessdot w$ is a pseudo descent with respect to the root $r$, then the label sequence of the first atom chain $c(r,u,w)$ with respect to $r$ and $\lambda$ strictly lexicographically precedes the label sequence of $u\lessdot v \lessdot w$ with respect to $r$ and $\lambda$. We will then show that $\lambda$ and $\Omega$ are compatible for an arbitrary rooted interval $[x,y]_r$ using \cref{prop:exists path to first atom chain}. 

    Let $u\lessdot v\lessdot w$ be a pseudo descent with respect to the root $r$ and $\Omega$. Let $v'$ be the element of $c(r,u,w)$ covering $u$, that is, $v'$ is the first atom of $[u,w]_r$. Note that $v'\neq v$. Just as in the proof of \cref{lem:lex cc from tcl gives shelling}, \cref{prop:same cc relable label seqs} (i) implies that $\lambda(r,u,v')\leq \lambda(r,u,v)$. Also, let $w'$ be the element of $c(r,u,w)$ covering $v'$. 

    If $\lambda(r,u,v')< \lambda(r,u,v)$, then the label sequence of the first atom chain $c(r,u,w)$ with respect to $r$ and $\lambda$ 
    strictly lexicographically precedes the label sequence of $u\lessdot v \lessdot w$ with respect to $r$ and $\lambda$. 
    Suppose then that $\lambda(r,u,v') = \lambda(r,u,v)$. It suffices to show that $\lambda(r\cup v',v',w')<\lambda(r\cup v,v,w)$. By the proof of \cref{prop:same cc relable label seqs} (ii), $\lambda(r\cup v',v',w')\neq \lambda(r\cup v,v,w)$ since $v'\neq v$. Assume seeking contradiction that $\lambda(r\cup v,v,w)<\lambda(r\cup v',v',w')$. We will produce three chains that violate (LC) of \cref{def:LCRFAS}, which contradicts the fact that $\Omega$ is an LCRFAS. \cref{fig:LCRFAS argument pic} shows a schematic of these chains that can be followed for the remainder of the argument. 
    Let $m_j$ be the earliest chain in $\Gamma$ that contains $r\cup \{v',w'\}$. Observe that since $\Gamma$ is a linear extension of $\preceq$ and $c(r\cup \{v',w'\}, w',\oneh)$ is the unique minimal element of $\preceq$ restricted to $[w',\oneh]_{r\cup \{v',w'\}}$ (by \cref{prop:compatible max chain order from RFAS}), we have that the portion of $m_j$ above $w'$ is $c(r\cup \{v',w'\}, w',\oneh)$. Further observe that for the same reasons, any maximal chain containing $r\cup  \{v',w'\}$ 
    follows $m_j$ in $\Gamma$. We will produce two maximal chains $m_i$ and $m_k$ that both contain $u\lessdot v\lessdot \tilde{w}$ for some $\tilde{w}$ such that $i<j<k$.
    
    Recall that by \cref{def:cc label from tcl}, $\lambda(r\cup v,v,w)=h$ where $h$ is the position of some maximal chain $m_h$ in $\Gamma$ that contains $r\cup v$. Notice that the label of the cover relation above $v$ in $m_h$ is $h$, otherwise $\lambda(r\cup v,v,w)$ would be the position of some earlier chain in $\Gamma$ than $m_h$. Observe that since $m_j$ is the first chain in $\Gamma$ containing $r\cup \{v',w'\}$, \cref{rmk:first chain label} implies that $\lambda(r\cup v',v',w')\leq j$. Since we have assumed $\lambda(r\cup v,v,w)<\lambda(r\cup v',v',w')$, we have $h=\lambda(r\cup v,v,w)<\lambda(r\cup v',v',w')\leq j$. Thus $m_h$ precedes $m_j$ in $\Gamma$. Now let $m$ be the earliest chain in $\Gamma$ that contains $r\cup \{v,w\}$ and let $d$ be the portion of $m$ above $w$. Observe that since $r\cup c(r,u,w) \cup d$ contains $r\cup \{v',w'\}$, we have $m_j\preceq r\cup c(r,u,w) \cup d$. Then since $r\cup c(r,u,w) \cup d \to m$, we have that $m_j$ precedes $m$ in $\Gamma$. Notice that in $\Gamma$, $m_h$ is followed by $m_j$ which is followed by $m$. Furthermore, $r\cup v$ is contained in both $m_h$ and $m$ and the labels of the cover relations above $v$ in both $m_h$ and $m$ are $h$. 
    
    Now let $m'$ be the first maximal chain after $m_j$ in $\Gamma$ that contains both $r \cup v$ and an element $\hat{w}$ covering $v$ such that $\lambda(r \cup v, v, \hat{w})=h$. We know $m'$ exists because $m$ satisfies these conditions. We consider two cases: either $m'$ is not the first chain in 
    $\Gamma$ containing $r\cup \{v,\hat{w}\}$ or $m'$ is the first chain containing $r\cup \{v,\hat{w}\}$. If $m'$ is not the first chain in 
    $\Gamma$ containing $r\cup \{v,\hat{w}\}$, then the first chain $m_i$ containing $r\cup \{v,\hat{w}\}$ precedes $m_j$ in $\Gamma$ by choice of $m'$. Consider $m_i$, $m_j$, and $m_k=m'$. We have $i<j<k$ and these three chains violate (LC) of \cref{def:LCRFAS} since $m_i$ and $m_k=m'$ both contain $r\cup \{v,\hat{w}\}$ while $m_j$ contains $r\cup \{v',w'\}$ and $v\neq v'$.

    On the other hand, suppose $m'$ is the first chain in 
    $\Gamma$ containing $r\cup \{v,\hat{w}\}$. We will find new maximal chains $m_i$ and $m_k$ that together with $m_j$ violate (LC). Since $\lambda(r\cup v,v,\hat{w}) = \lambda(r\cup v,v,w)=h$ and $m'\neq m_h$,
    \cref{rmk:not labeled by your first chain} implies that there is an element $\tilde{w}$ that covers $v$ and satisfies the following properties: there exist maximal chains $m_i$ and $m_k$, both containing $r\cup \{v,\tilde{w}\}$, such that $m_{i}$ precedes $m'$ and $m'$ precedes $m_k$ in $\Gamma$. This means that $\lambda(r\cup v,v,\tilde{w})=\lambda(r\cup v,v,\hat{w}) = \lambda(r\cup v,v,w)=h$. By choice of $m'$, we have that $m_{i}$ precedes $m_j$ in $\Gamma$. We also have that $m_j$ precedes $m'$ which precedes $m_k$. So, we have $i<j<k$ and $m_i$, $m_j$, $m_k$ violate (LC) of \cref{def:LCRFAS} since $m_i$ and $m_k$ both contain $r\cup \{v,\tilde{w}\}$ while $m_j$ contains $r\cup \{v',w'\}$ and $v\neq v'$. Therefore, $\lambda(r\cup v',v',w')<\lambda(r\cup v,v,w)$.

    Now we show that $\lambda$ is compatible with $\Omega$. Suppose $m'$ is a maximal chain in the rooted interval $[x,y]_{r}$ that is not the first atom chain. By \cref{rmk: restriction of rfas is rfas} and \cref{prop:exists path to first atom chain}, there is a sequence $d_1,d_2,\dots, d_s$ of maximal chains in $[x,y]_{r}$ such that $r\cup c(r,x,y)\to r\cup d_1\to r\cup d_2\to \dots \to r\cup d_s \to r\cup m'$. Recall from \cref{def:max chain order from RFAS} that each step in the sequence corresponds to replacing a first atom chain with a single pseudo descent. By the argument above the label sequences strictly increase lexicographically at each step.
\end{proof}

We are now ready to prove the main result of the section. 

\begin{theorem}\label{thm: tcl iff RFAS}
  A bounded poset $P$ admits an LCRFAS if and only if $P$ is TCL-shellable. 
\end{theorem}

\begin{proof} 
We begin by proving that if $P$ admits an LCRFAS $\Omega$, then $P$ is TCL-shellable. Let $\lambda$ be a CE-labeling of $P$ that is compatible (see Definition \ref{def:compatible RFAS}) with $\Omega$; such a labeling exists by Lemma \ref{lem: RFAS has compatible labeling}.

We have that $\lambda$ is a CE-labeling of $P$ by integers, and no two maximal chains have the same label sequence by \cref{prop:same cc relable label seqs} (ii). We must show that for each interval $[x, y]_r$, there is a unique, topologically ascending chain and it is lexicographically first. 

By the final paragraph of the proof of \cref{lem: RFAS has compatible labeling}, the first atom chain $c(r,x,y)$ has the lexicographically smallest label sequence with respect to $r$ and $\lambda$ among all maximal chains of $[x,y]_r$. Thus, $c(r,x,y)$ is topologically ascending with respect to $\lambda$. Further, this implies that every pseudo descent with respect to $\Omega$ (see \cref{def:pseudo descent}) is a topological descent with respect to $\lambda$. Now by \cref{prop: not first atom then have pseudo descent}, any maximal chain $m$ of $[x,y]_r$ that is distinct from $c(r,x,y)$ contains a pseudo descent, and so contains a topological descent with respect to $\lambda$. Hence, $c(r,x,y)$ is the unique topologically ascending chain in the rooted interval and has the lexicographically smallest label sequence. Therefore, $\lambda$ is a TCL-labeling.

Lemma \ref{lem:TCL implies RFAS} is the reverse direction.  
\end{proof}

\begin{remark}\label{rmk:only need compatible with RFAS to get TCL}
Observe that in the proof of \cref{thm: tcl iff RFAS} the only place where condition (LC) of an LCRFAS (\cref{def:LCRFAS}) is used is to produce a labeling that is compatible with the LCRFAS. Once we know there exists a labeling $\lambda$ that is compatible with the LCRFAS, only the RFAS conditions (\cref{def:RFAS}) are necessary to show that $\lambda$ is a TCL-labeling. In particular, only the RFAS conditions are necessary to show that every pseudo descent is a topological descent with respect to $\lambda$. Thus, any labeling that is compatible with an RFAS, in the sense of \cref{def:compatible RFAS}, is a TCL-labeling. However, without condition (LC), there may not exist a labeling that is compatible with an RFAS as the following \cref{ex:RFAS not LCRFAS} shows.
\end{remark}

\begin{example}\label{ex:RFAS not LCRFAS}
    Consider the poset $P$ shown in \cref{fig:RFAS not LCRFAS}. There is an RFAS $\Omega$ of $P$ for which there does not exist a labeling compatible with $\Omega$. This shows an LCRFAS rather than an RFAS is necessary in \cref{thm: tcl iff RFAS}. Define $\Omega$ as follows. In each rooted interval $[x,y]_r$, the first atom $\Omega(r,x,y)$ is the leftmost atom in the interval when $P$ is drawn as in \cref{fig:RFAS not LCRFAS} except for the following: $\Omega(\{\zh\},\zh,e)=b$, $\Omega(\{\zh\},\zh,i)=c$, $\Omega(\{\zh\},\zh,n)=b$, $\Omega(\{\zh,b\},b,j)=i$, $\Omega(\{\zh,c\},c,k)=g$, $\Omega(\{\zh,c\},c,\oneh)=g$, and $\Omega(r,f,\oneh)=k$ for both roots $r$ of $f$. 

    While it is somewhat tedious to verify that this is an RFAS, it is fairly straightforward to verify that the partial order $\preceq$ on the maximal chains of $P$ induced by $\Omega$ as in \cref{def:max chain order from RFAS} contains the chain $\{\zh,b,d,r,\oneh\}\to \dots \to \{\zh,c,i,k,\oneh\} \to \{\zh,c,i,j,\oneh\} \to \{\zh,b,i,j,\oneh\} \to \{\zh,b,d,j,\oneh\}$. Observe that $\Omega$ violates (LC) of \cref{def:LCRFAS} since any linear extension of $\preceq$ has some chain containing $\{\zh,b,d\}$ coming before a chain containing $\{\zh,c,i\}$ which comes before another chain containing $\{\zh,b,d\}$. Further, by \cref{rmk:only need compatible with RFAS to get TCL}, any CE-labeling $\lambda$ that is consistent with $\Omega$ is a TCL-labeling. Also, as observed in the proof of \cref{thm: tcl iff RFAS}, every pseudo descent with respect to $\Omega$ is a topological descent with respect to $\lambda$. Thus, the lexicographic order on maximal chains of $P$ induced by $\lambda$ is at least as fine as $\preceq$. Then by the proof of \cref{thm: tcl iff cc}, there is a CC-labeling $\lambda'$ of $P$ by integers obtained from $\lambda$ via \cref{def:cc label from tcl} such that the topological descents with respect to $\lambda'$ are the same as the topological descents with respect to $\lambda$. Thus, the lexicographic order induced by $\lambda'$ is also at least as fine as $\preceq$. Now by \cref{prop:labels of out and back roots}, the label sequences of $\{\zh,b,d\}$, $\{\zh,b,i\}$ and $\{\zh,c,i\}$ assigned by $\lambda'$ are the same. This contradicts that $\lambda'$ is a CC-labeling since $\{\zh,b,i\}$ and $\{\zh,c,i\}$ are the only maximal chains in the interval $[\zh,i]$, so $[\zh,i]$ does not have a unique topologically ascending maximal chain. (Notice that this also contradicts \cref{prop:same cc relable label seqs} (ii).) Therefore, there is no CE-labeling that is consistent with $\Omega$. This implies that (LC) of \cref{def:LCRFAS} is necessary for the existence of a CE-labeling that is compatible with an RFAS. 
    \end{example}

\begin{figure}
\begin{center}
\begin{tikzpicture}[scale=0.5]

\node at (5,0){$\mathbf{\zh}$}; 
\node at (0,3){$\mathbf{a}$};

\node at (5,3){$\mathbf{b}$};

\node at (10,3){$\mathbf{c}$};

    \node at (-2,6){$\mathbf{d}$};

    \node at (1,6){$\mathbf{e}$};

\node at (4,6){$\mathbf{f}$};

\node at (6,6){$\mathbf{g}$};

    \node at (9,6){$\mathbf{h}$};

\node at (12,6){$\mathbf{i}$};

\node at (0,9){$\mathbf{r}$};

\node at (3,9){$\mathbf{k}$};

\node at (7,9){$\mathbf{n}$};

\node at (10,9){$\mathbf{j}$};

\node at (5,12){$\mathbf{\oneh}$};


\draw (4.8,0.2)--(0,2.8);
\draw (5,0.3)--(5,2.8);
\draw (5.2,0.2)--(10,2.8);

\draw (-0.2,3.1)--(-2,5.8);
\draw (0,3.2)--(0.9,5.8);
\draw (0.2,3)--(8.8,5.9);

\draw (4.8,3)--(-1.9,5.8);
\draw (4.8,3.1)--(1.1,5.8);
\draw (5,3.2)--(4,5.7);
\draw (5.1,3.2)--(5.9,5.8);
\draw (5.2,3)--(11.9,5.9);

\draw (9.8,3)--(4.2,6);
\draw (9.8,3.1)--(6.2,5.9);
\draw (10,3.2)--(9.1,5.8);
\draw (10.2,3)--(12,5.8);

\draw (-2,6.2)--(-0.2,8.9);
\draw (-1.8,6.1)--(2.8,8.9);
\draw (-1.8,6)--(9.8,9);

\draw (1,6.2)--(2.9,8.8);
\draw (1.2,6.1)--(6.8,8.9);

\draw (3.8,6.1)--(0.1,8.9);
\draw (4,6.3)--(3,8.8);

\draw (5.8,6.1)--(3.1,8.8);

\draw (8.8,6.1)--(3.2,8.8);
\draw (8.9,6.2)--(7.1,8.8);
\draw (9.1,6.2)--(10,8.7);

\draw (11.8,6.1)--(3.2,8.9);
\draw (11.9,6.2)--(10.1,8.9);

\draw (0.1,9.2)--(4.8,11.9);
\draw (3,9.2)--(4.9,11.7);
\draw (7,9.2)--(5.1,11.7);
\draw (9.9,9.2)--(5.2,11.9);

\end{tikzpicture}

\caption{A poset that admits an RFAS $\Omega$ such that there is no CE-labeling compatible with $\Omega$.}
\label{fig:RFAS not LCRFAS}
\end{center}
\end{figure}
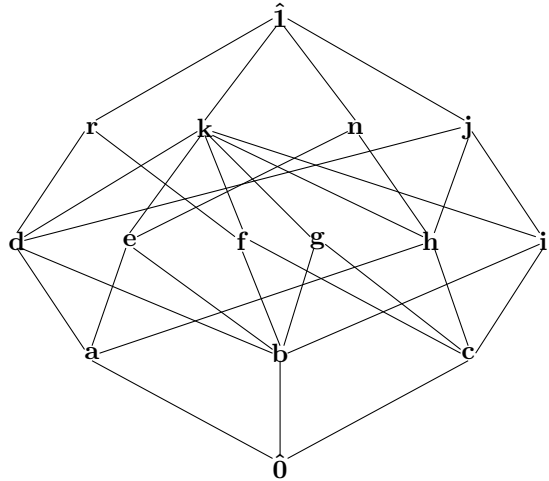

\section{Open Questions}\label{sec:open questions}
A natural question that arises in light of this work is whether there exist posets that admit RFASs but that do not admit LCRFASs. To this end, one may consider looking at the examples presented in \cite{vincewachsshellnonlexshell} and \cite{walkershellnonlexshell}, which are shellable but not CL-shellable and may or may not be CC-shellable. Another question that remains is whether there exists a graded poset that is CC-shellable but not EC-shellable. It is possible that the graded example presented in \cite{liclnotel} is one such example, but the method used there for proving this poset is not EL-shellable cannot be easily adapted, at least as far as we can see, to show that the poset is not EC-shellable.
        
\bibliographystyle{amsplain}
\bibliography{ct-hershstadnyk}
\end{document}